\documentclass[12pt,a4paper]{article}
\usepackage{graphicx}
\usepackage[]{amsmath,amssymb,amsthm}
\usepackage{paralist}
\usepackage{epstopdf}
\usepackage{float}

\newtheorem{theorem}{Theorem}[section]
\newtheorem{proposition}[theorem]{Proposition}

\newtheorem{lemma}[theorem]{Lemma}
\newtheorem{corollary}[theorem]{Corollary}
\theoremstyle{definition}
\newtheorem{definition}[theorem]{Definition}
\theoremstyle{remark}

\newcommand{\R}{\mathbb{R}}

\newcommand{\Z}{\mathbb{Z}}

\newcommand{\inn}{\textnormal{in}}
\newcommand{\out}{\textnormal{out}}

\usepackage{tikz}
\usetikzlibrary{decorations.pathreplacing,decorations.markings,positioning,shapes.misc}
\usepackage{hyperref}
\usetikzlibrary{calc}

\tikzset{
    middlearrow/.style n args={3}{
        draw,
        decoration={
            markings,
            mark=at position 0.5 with {
                \arrow[scale=2]{#1};
                \path[#2] node {$#3$};
            },
        },
        postaction=decorate
    }
}

\tikzset{cross/.style={cross out, draw=black, minimum size=20*(#1-\pgflinewidth), inner sep=0pt, outer sep=0pt},
cross/.default={1pt}}

\begin{document}

\begin{center}
{\large{\bf 
Complete heteroclinic networks derived from graphs consisting of two cycles}}\\
\mbox{} \\
\begin{tabular}{cc}
{\bf Sofia B.\ S.\ D.\ Castro$^{\dagger}$} & {\bf Alexander Lohse$^{\ddagger,*}$} \\
{\small sdcastro@fep.up.pt} & {\small alexander.lohse@uni-hamburg.de}
\end{tabular}

\end{center}

\noindent $^{*}$ Corresponding author.

\noindent $^{\dagger}$ Faculdade de Economia and Centro de Matem\'atica, Universidade do Porto, Rua Dr.\ Roberto Frias, 4200-464 Porto, Portugal.

\noindent $^{\ddagger}$ Fachbereich Mathematik, Universit\"at Hamburg, Bundesstra{\ss}e 55, 20146 Hamburg, Germany

\begin{abstract}
We address the question how a given connection structure (directed graph) can be realised as a heteroclinic network that is complete in the sense that it contains all unstable manifolds of its equilibria. For a directed graph consisting of two cycles we provide a constructive method to achieve this: (i) enlarge the graph by adding some edges and (ii) apply the simplex method to obtain a network in phase space. Depending on the length of the cycles we derive the minimal number of required new edges. In the resulting network each added edge leads to a positive transverse eigenvalue at the respective equilibrium. We discuss the total number of such positive eigenvalues in an individual cycle and some implications for the stability of this cycle.
\end{abstract}



\vspace{.3cm}

\vspace{2cm}

\section{Introduction}
It is well-known that a heteroclinic network can be represented as a directed graph (digraph) by identifying equilibria and connections in the network with vertices and edges in the digraph. Every such digraph is strongly connected, i.e., there is a directed path between any given pair of vertices. Conversely, given a digraph several methods exist to design a dynamical system, defined through a differential equation, exhibiting a heteroclinic network corresponding to the digraph.

We are interested in using the simplex method from~\cite{AshPos2013} to construct complete heteroclinic networks whose corresponding digraphs contain a given subgraph that consists of two cycles. As proved in~\cite{AshPos2013}, this method can realise any strongly connected digraph $G$ without 1- and 2-cycles as a heteroclinic network $X$ such that each equilibrium in $X$ lies on its own coordinate axis and connections are contained in coordinate planes. The required space dimension is equal to the number of vertices in the graph.

A heteroclinic network is called complete if it contains the entire unstable manifolds of all its equilibria. This notion is not to be confused with that of a complete graph, where there is an edge between any pair of vertices. Completeness is a necessary condition for asymptotic stability of the network and therefore of interest for the long-term ``visibility'' of a dynamical structure. Typically, a network resulting from the simplex method is not complete. This is because any equilibrium with two or more outgoing connections will have an unstable manifold of dimension at least two, and only in special cases will all initial conditions within this manifold limit to other equilibria in the network.

Given a strongly connected digraph $G$, we therefore ask whether or not there is another digraph $G'$ with $G \subset G'$ such that the realisation of $G'$ by the simplex method is a complete heteroclinic network. In general, there are two mechanisms to enlarge $G$ and possibly obtain a suitable $G'$: (i) adding a new vertex, (ii) adding a new edge between existing vertices. Both can be used finitely many times to find a suitable $G'$.

In this paper we are only concerned with adding edges, i.e.\ we do not introduce new vertices in the prescribed connection structure. For related results regarding the addition of vertices see e.g.~\cite{AshCasLoh2020}. (Note that adding a vertex also implies adding edges that connect this vertex to the others, since $G'$ should be a strongly connected digraph.) We therefore ask more precisely:

\begin{itemize}
\item[(Q1)] Given a strongly connected digraph $G$ without 1- and 2-cycles, can we add edges between the vertices in $G$ to obtain a digraph $G'$ such that the realisation of $G'$ through the simplex method produces a complete heteroclinic network $X$?
\end{itemize}

If the answer is yes, we call $X$ a \emph{complete realisation of $G \subset G'$}. 
The vector field for this realisation is obtained by equation (2) in ~\cite{AshPos2013} for the graph $G'$. Additionally, we may then ask:

\begin{itemize}
\item[(Q2)] What is the minimal number of edges that have to be added in $G$ to obtain a suitable $G'$ that yields a complete realisation?
\end{itemize}

We call a realisation $X$ of $G'$ a \emph{minimal complete realisation of $G \subset G'$} if the number of edges added to $G$ to obtain $G'$ is minimal. Note that (minimal) realisations need not be unique. Also, a heteroclinic network $X$ is typically a realisation of many different subgraphs $G \subset G'$ (corresponding to subcycles/subnetworks of $X$).

Since we go back and forth between looking at digraphs and the corresponding heteroclinic networks generated by the simplex method we use the figures in this paper both as representing graphs and networks. When thinking about graphs, we talk about vertices and edges, whereas for networks we use equilibria and connections, respectively.

In this paper we are concerned with adding edges/connections to obtain complete realisations for a particular class of digraphs $G$, namely those which are the connected union of two cycles $C_1$ and $C_2$. For such a graph $G$, the two cycles must have at least one common vertex. They may or may not have one or more common edges. If they share more than one edge, then all such edges must occur consecutively without interruption by edges that belong to only one cycle, because otherwise there would be more than two cycles, see Figure~\ref{fig:donut-simple}. A list of examples from the literature for networks that are the union of two cycles is provided in Appendix~\ref{app:list}. 

We provide full answers to questions (Q1) and (Q2) for this class of digraphs/networks. Moreover, we note that the answer to (Q1) is not always affirmative for digraphs with more than two cycles as shown in section~\ref{sec:further}.

As explained above we use the simplex method to create the desired heteroclinic networks from such graphs. In particular, for the resulting systems all coordinate axes are dynamically invariant, the equilibria lie on coordinate axes and the coordinate planes contain heteroclinic trajectories.

As an example, consider the well-known Kirk and Silber network~\cite{KirSil1994}, shown in Figure~\ref{fig-KS-Bowtie-House} left. It is not complete, because the two-dimensional unstable manifold of $\xi_2$ is not entirely contained in the network. However, it can be embedded in a complete network with an additional equilibrium as discussed in~\cite{AshCasLoh2020}. This bigger network contains additional cycles and is of a different nature, since the added equilibrium is not on a coordinate axis. Thus, the network cannot be thought of as generated by the simplex method.

Alternatively, a connection between $\xi_3$ and $\xi_4$ can be added in the Kirk and Silber network, creating the $(B_3^-,B_3^-,C_4^-)$ network from~\cite{CasLoh2016b}. This network first appeared in Figure~4(b) in~\cite{Bra1994}. It contains three cycles and is a minimal complete realisation of the graph corresponding to the Kirk and Silber network in the above sense.

This example prompts a brief comparison of the method presented here and that of Ashwin {\em et al.} \cite{AshCasLoh2020}: The present method adds edges to the graph, transforming the original structure into a more connected one. That of \cite{AshCasLoh2020} adds vertices and edges that are typically not thought as becoming part of the network but are seen as a side structure. Thus, the networks of interest in \cite{AshCasLoh2020} are not complete but rather ``almost complete''.  They can be seen as the outcome of a pitchfork bifurcation at one of the equilibria in a network created by the addition of edges. This bifurcation produces an excitable network, see \cite{AshPos2015}.
\medbreak

Finally, in addition to the questions (Q1) and (Q2) outlined above, it is natural to also ask the following:

\begin{itemize}
\item[(Q3)] For a digraph consisting of two cycles, how do the answers to (Q1) and (Q2) change if we wish to prescribe which of the two resulting heteroclinic cycles is the more stable one?
\end{itemize}

In section~\ref{sec:stability} we address this by deriving the number of positive transverse eigenvalues created in a cycle through the choice of a particular complete realisation. Furthermore, we explain the effect of those eigenvalues on the stability of the cycle. This interest in stability justifies the restriction of our attention to heteroclinic networks realised by the simplex method of~\cite{AshPos2013}: In fact, the simplex method produces quasi-simple heteroclinic cycles and networks and there are suitable tools to study their stability, see e.g.~\cite{GarCas2019}. We recall that a heteroclinic cycle that is part of a heteroclinic network can never be asymptotically stable, see~\cite{PodCasLab2019}.
More detail is provided in the next section.

In the context of coupled cell networks several authors are interested in studying the effect of adding edges in the preservation of synchrony. We mention the work of~\cite{MilSunNis2010} and \cite{PoiPadPer2019} as examples where the adjacency matrix is changed by the creation of edges between cells, leading (or not) to changes in synchrony. Our addition of edges is different and represents the addition of a connection between two equilibria for the dynamics. Since we do not have internal dynamics, the issue of synchrony does not arise in our setting.
An interesting account of heteroclinic networks in the context of coupled cell networks can be found in~\cite{Fie2015}. We do not explore here the relation between adding edges to a network of coupled cells  and the addition of connections in heteroclinic networks.

The task of realising a digraph as a heteroclinic network is of interest from a modelling perspective. Our method introduces no additional nodes in the graph, only edges. This absence of additional equilibria in the resulting heteroclinic network is important, e.g., in the context of population dynamics, where a new equilibrium corresponds to a population distribution that was not represented in the original configuration. Preservation of the stability of the cycles in the original network, again in the context of population dynamics or game theory, corresponds to preserving the surviving species or the actions. See \cite{CasFerLab2024} for a concrete example. Another potential field of application is that of social networks or opinion dynamics where the number of connections plays an important role.

The rest of this paper is structured as follows: in section~\ref{sec:prelim} we introduce some concepts from graph theory and the study of heteroclinic dynamics that are required for our investigations. In section~\ref{sec:completions} we prove our results on the existence of (minimal) complete realisations of graphs consisting of two cycles before we address the stability of the cycles in the created networks in section~\ref{sec:stability}. Section~\ref{sec:further} contains some comments about generalizations and limitations of our method to create complete networks, before section~\ref{sec:conclusion} concludes.

\section{Preliminaries}\label{sec:prelim}
We need some definitions from graph theory and some from the study of heteroclinic networks.

A digraph $G=(V,A)$ is given by a set of vertices $V=\{v_1, \ldots ,v_m\}$ and directed edges $[v_j \to v_i] \in A$ connecting them. A subset of $k$ distinct vertices in $V$ connected in a cyclic way is called a cycle of length $k$ or a $k$-cycle.

The out-degree of a vertex $v_j \in V$ is the number of outgoing directed edges $[v_j \to v_i] \in A$. We call a vertex with an out-degree greater than 1 a \emph{distribution vertex}. Similarly, the in-degree of $v_j \in V$ is the number of incoming directed edges $[v_i \to v_j] \in A$, and we call a vertex with an in-degree greater than 1 a \emph{collection vertex}.
 
Given an ordinary differential equation $\dot x =f(x)$ on $\R^n$ with a smooth vector field $f$ and a finite set of hyperbolic equilibria $\xi_1, \ldots ,\xi_m$, a non-empty intersection of invariant manifolds $W^u(\xi_j) \cap W^s(\xi_i) \neq \emptyset$ is called a heteroclinic connection from $\xi_j$ to $\xi_i$. We also denote it by $[\xi_j \to \xi_i]$.
\footnote{When $\xi_j=\xi_i$ we say the connection is homoclinic.}

A collection $X$ of such equilibria and heteroclinic connections can be associated with a directed graph by drawing a vertex for each equilibrium and an edge for each heteroclinic connection. If the resulting digraph is a cycle, we call $X$ a heteroclinic cycle. A heteroclinic network is a connected union of finitely many heteroclinic cycles, and thus associated with a strongly connected digraph containing more than one cycle. It is sometimes important to carefully distinguish between the \emph{full set of connections} $W^u(\xi_j) \cap W^s(\xi_i)$ and a \emph{single connecting trajectory} within it. For this section and the next we focus on the full sets of connections, while in Section~\ref{sec:stability} we give some results towards understanding the stability of subcycles consisting of specific choices of connecting trajectories.

As mentioned in the introduction, for a given digraph $G$ without 1- and 2-cycles, the simplex method in \cite{AshPos2013} provides a polynomial vector field $f:~\R^n~\to~\R^n$ such that the dynamical system given through
\[\dot x = f(x) \]
contains a heteroclinic network $X$ that corresponds to $G$ in the above sense. If $G$ has $n$ vertices, then the required space dimension is also $n$, as the simplex method realises each vertex in $G$ as an equilibrium on its own coordinate axis. Moreover, there always exists a connecting trajectory between two equilibria that is contained in the coordinate plane spanned by their respective directions. The resulting system has $\Z_2^n$ symmetry, with $\Z_2$ acting by reflection across each coordinate hyperplane. Thus, all coordinate subspaces are flow-invariant.

If an equilibrium $\xi_j \in X$ is associated with a distribution (collection) vertex in the digraph $G$, we call $\xi_j$ a \emph{distribution (collection) node in $X$}. By construction, the out-degree of a vertex in $G$ typically equals the dimension of the unstable manifold of the corresponding equilibrium in $X$. Therefore, a distribution node has an unstable manifold of dimension two or higher.

Distribution vertices/nodes may give rise to so-called $\Delta$-cliques: in the graph, a $\Delta$-{\em clique} is a triangle that is not strongly connected, see Definition 2.2 in \cite{AshCasLoh2020}. Thus, any $\Delta$-clique contains a distribution vertex with two outgoing edges, see Figure~\ref{fig-Delta-clique-graph}. Podvigina {\em et al.} \cite{PodCasLab2020} call such a vertex a $\emph{b-point}$ of the $\Delta$-clique.

In the corresponding heteroclinic network, a $\Delta$-clique leads to a two-dimensional unstable manifold that is bounded, we refer to Definition 2.1 in \cite{PodCasLab2020}: all trajectories in the unstable manifold of the distribution node in the $\Delta$-clique converge to one of the two other equilibria. In Figure~\ref{fig-Delta-clique-net}, we show the two-dimensional unstable manifold of $\xi_2$ in the Kirk-Silber network: depending on the choice of parameters, it may be unbounded (left), form a $\Delta$-clique $\Delta_{234}$ (middle) or an additional equilibrium appears which connects to $\xi_3$ and $\xi_4$ (right). In the case of a $\Delta$-clique, the entire unstable manifold of $\xi_2$ in $\Delta_{234}$ limits to the set $\{\xi_3, \xi_4\}$ which is a subset of the network. 

For a $\Delta$-clique in the graph, the simplex method generates a $\Delta$-clique in the corresponding heteroclinic network.

\begin{figure}[!htb]
 \centerline{
\begin{tikzpicture}[xscale=1, yscale=1]
\filldraw[black] (0,2.5) circle (2pt) node[above] {\small$\xi_2$};
\filldraw[black] (-2,0) circle (2pt) node[left] {\small$\xi_1$};
\filldraw[black] (2,0) circle (2pt) node[right] {\small$\xi_3$};
\path (-2,0) edge[black,thick, middlearrow={>}{}{}] (0,2.5);
\path (0,2.5) edge[black,thick, middlearrow={>}{}{}] (2,0);
\path (-2,0) edge[black,thick, middlearrow={>}{}{}] (2,0);
\end{tikzpicture}
}
 \caption{A $\Delta$-clique $\Delta_{123}$ in a graph with distribution vertex $\xi_1$.\label{fig-Delta-clique-graph}}
 \end{figure}

\begin{figure}[!htb]
\centerline{\includegraphics[width=4cm]{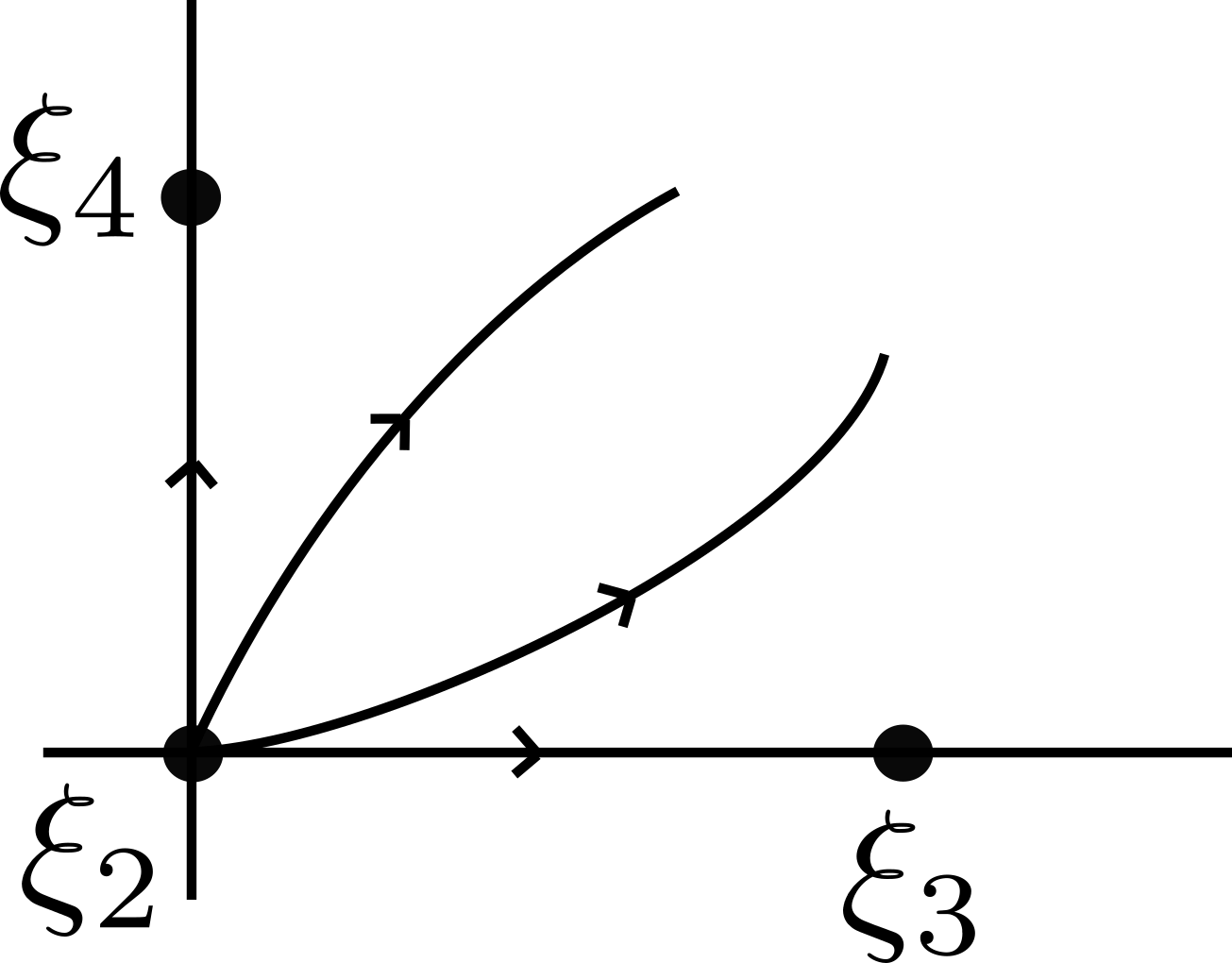} \quad \includegraphics[width=4cm]{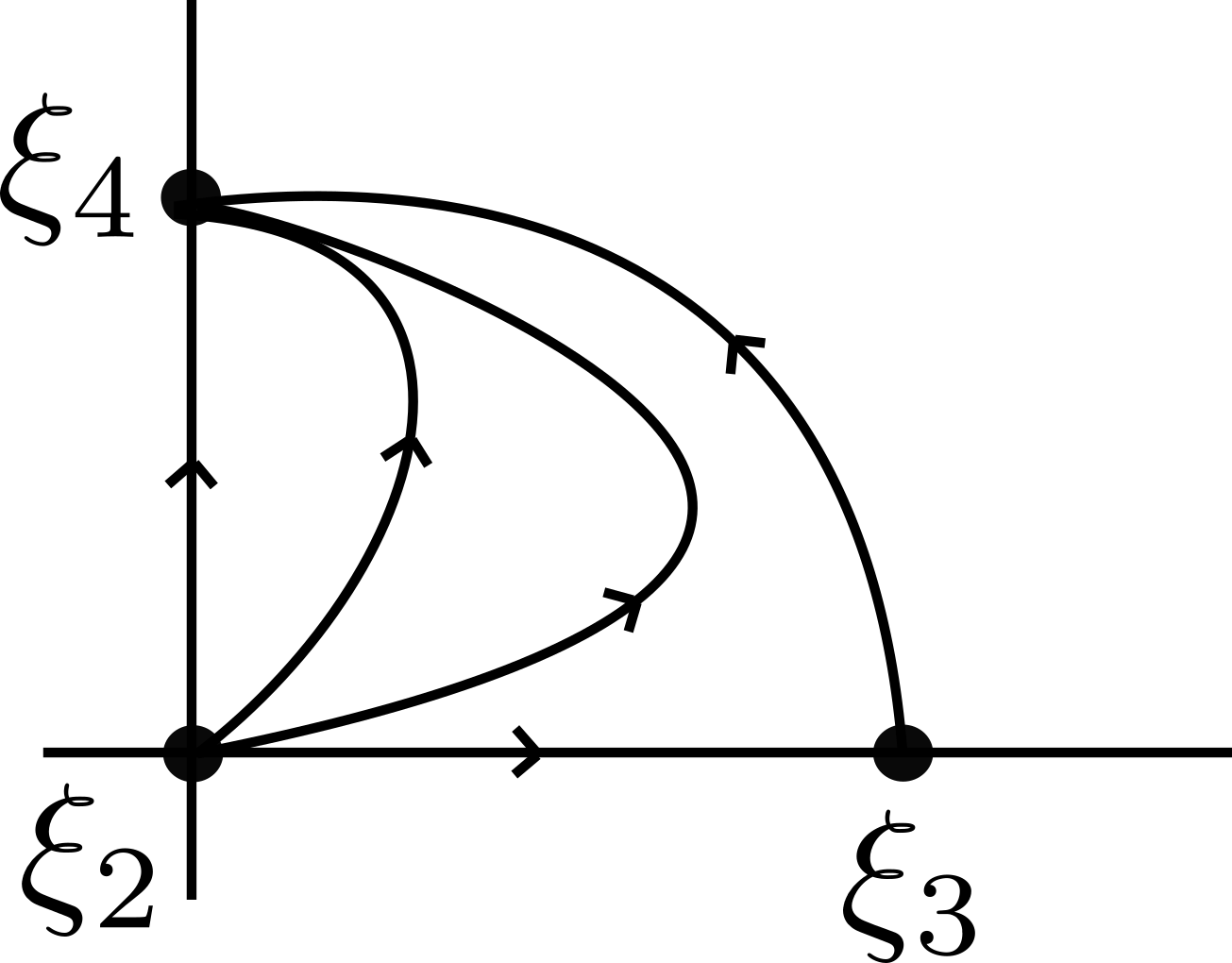} \quad \includegraphics[width=4cm]{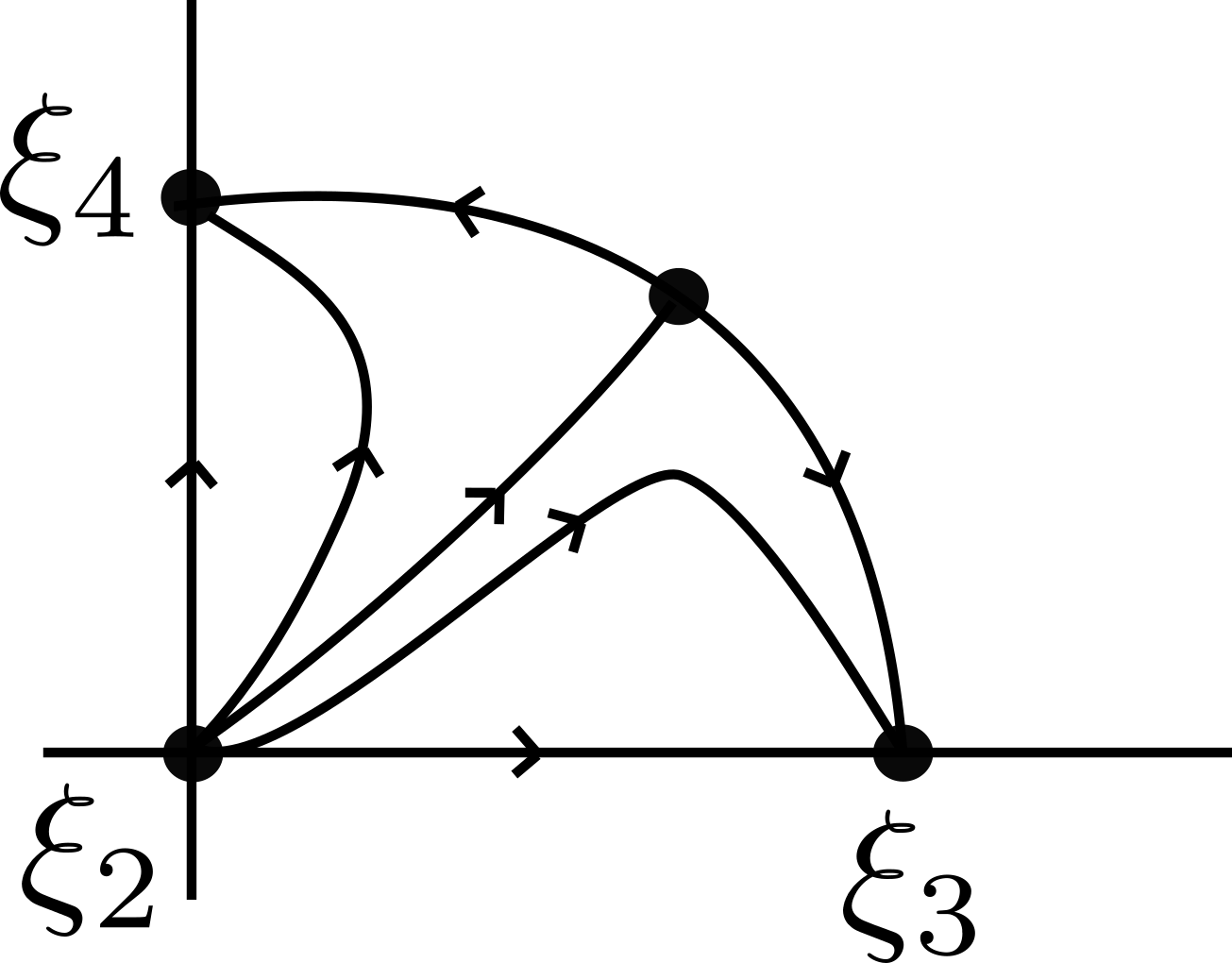}}
\caption{A node with two outgoing connections gives rise to a two-dimensional unstable manifold $W^u(\xi_2)$ in the corresponding heteroclinic structure (left). In the middle, $W^u(\xi_2)$ is bounded by being included in the $\Delta$-clique $\Delta_{234}$. On the right, $W^u(\xi_2)$ is bounded by the unstable manifold of an additional equilibrium.}\label{fig-Delta-clique-net}
\end{figure}

A heteroclinic network is called \emph{complete}, see \cite{AshCasLoh2020}, if it contains the unstable manifolds of all its equilibria. From the paragraph above, it follows immediately that $\Delta$-cliques play a crucial role for the completeness of networks generated by the simplex method. The method we introduce below ensures that all distribution nodes are b-points of $\Delta$-cliques.

An undirected graph is called \emph{complete} if for every pair of its vertices there is an edge between these vertices. If these vertices are given an orientation (turning the undirected graph into a digraph), then the graph is called a \emph{tournament}. The notions of completeness for heteroclinic networks and graphs are independent in the following sense:

\begin{itemize}
	\item[(a)] A complete heteroclinic network does not necessarily correspond to a directed graph that is a tournament. An example is the network in Example 1 of \cite{PodCasLab2020}, see Figure~\ref{fig-complete-no-tournament}.
	\item[(b)] A graph that is a tournament does not necessarily produce a complete heteroclinic network when realised through the simplex method. This is explained in Section~\ref{sec:further}, see Figure~\ref{fig-3d-unstable-manifold}.
\end{itemize}

\begin{figure}[!htb]
 \centerline{
\begin{tikzpicture}[xscale=1, yscale=1]
\filldraw[black] (0,4) circle (2pt) node[above] {\small$\xi_3$};
\filldraw[black] (0,0) circle (2pt) node[below] {\small$\xi_1$};
\filldraw[black] (-2,2) circle (2pt) node[left] {\small$\xi_4$};
\filldraw[black] (2,2) circle (2pt) node[right] {\small$\xi_2$};
\path (0,0.3) edge[out=-90,in=90, black,thick, middlearrow={>}{}{}] (0,3.7);
\path (0,4) edge[black,thick, middlearrow={>}{}{}] (-2,2);
\path (2,2) edge[black,thick, middlearrow={>}{}{}] (0,4);
\path (-2,2) edge[black,thick, middlearrow={>}{}{}] (0,0);
\path (0,0) edge[black,thick, middlearrow={>}{}{}] (2,2);
\end{tikzpicture}
}
 \caption{The graph of the network in Example 1 of \cite{PodCasLab2020} is not a tournament (there is no edge between the vertices $\xi_2$ and $\xi_4$) even though the network is complete: only at $\xi_1$ is there a 2-dimensional unstable manifold which is captured by the $\Delta$-clique $\Delta_{123}$.\label{fig-complete-no-tournament}}
 \end{figure}
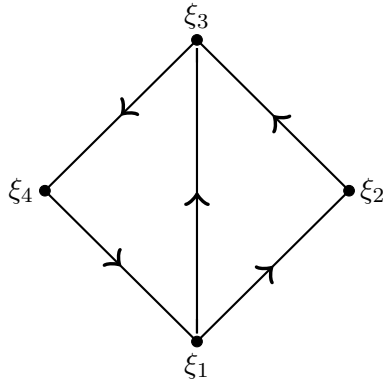

We say that a network is {\em quasi-simple} if all its cycles are quasi-simple. The definition of a quasi-simple cycle is as follows.

\begin{definition}[Definition 2.1 in \cite{GarCas2019}]\label{def:quasi-simple}
A {\em quasi-simple} cycle is a robust heteroclinic cycle connecting $m < \infty$ equilibria $\xi_j \in P_j\cap P_{j-1}$ so that $[\xi_j \rightarrow \xi_{j+1}] \subset P_j$ is one-dimensional, $P_j$ is flow-invariant, and $\mbox{dim }P_j=\mbox{dim }P_{j+1}$ for all $j=1, \hdots,m$.
\end{definition}

Quasi-simple cycles extend the notion of simple cycles introduced by \cite{KruMel1995} and preserve the use of their notation. In simple cycles, the connections lie in coordinate planes, hence the use of the letter $P$; 
the equilibria are on one-dimensional axes, hence the use of $L$ for line. Quasi-simple cycles exist with equilibria on the vertices of a product of simplices (an example is given in \cite{GarCas2019}) and the connections exist in affine space. To include this we define $\hat{L}_j$ to be the space connecting the equilibrium $\xi_j$ to the origin. In case of simple cycles $\hat{L}_j$ coincides with $L_j=P_{j-1} \cap P_j$. When a $\Delta$-clique is present, that is, when there are two-dimensional as well as one-dimensional connections, we consider the quasi-simple cycle whose trajectories are contained in the space of lowest possible dimension.

The eigenvalues of the linearisation of the right hand side of the system at the equilibria are important. As in \cite{GarCas2019}, we group them into four types, according to the location of the corresponding eigenvectors:
\begin{itemize}
	\item  {\em radial} eigenvalues $(-r<0)$, which have eigenvectors in $\hat{L}_j$;
	\item  {\em contracting} eigenvalues $(-c<0)$, which have eigenvectors in 
	$P_{j-1}$ but not in $\hat{L}_j$;
	\item  {\em expanding} eigenvalues $(e>0)$, which have eigenvectors in 
	$P_{j}$ but not in $\hat{L}_j$;
	\item  {\em transverse} eigenvalues $(t\in\R)$, otherwise.
\end{itemize}
We often use indices to indicate the equilibrium and connection they refer to. We use, for instance, $-c_{j}$ to denote a contracting eigenvalue at $\xi_j$ and $e_{j}$ to denote an expanding eigenvalue at $\xi_j$. Since the connections in a quasi-simple network are one-dimensional there is only one negative contracting and one positive expanding eigenvalue at each equilibrium. There may be more than one transverse eigenvalue which we denote by $t_{j,1}, \hdots, t_{j,n_t}$ where $n_t$ is the number of transverse eigenvalues.

Stability is an important and interesting issue when studying dynamics, in particular, dynamics near a heteroclinic cycle or network. The first well-known observation is that a heteroclinic cycle in a network cannot be asymptotically stable, see~\cite[Theorem 3.1]{PodCasLab2019}. That is, it cannot attract all nearby initial conditions. This is because the distribution node belongs to at least two cycles and therefore, there are initial conditions near the distribution node that are taken to (at least) two different cycles. The strongest notion of stability that applies is that of {\em essential asymptotical stability}, e.a.s.\ for short. Roughly speaking, an invariant object is e.a.s.\ if it attracts a set of points of almost full measure in a sufficiently small neighbourhood. See Melbourne \cite{Mel1991} for the original definition and Brannath \cite{Bra1994} for the current reinterpretation.  A weaker notion of stability, which includes e.a.s., is that of {\em fragmentary asymptotic stability}, f.a.s.\ for short, introduced by Podvigina \cite{Pod2012}. 
An invariant object is f.a.s.\ if it attracts a set of points of positive measure in a sufficiently small neighbourhood. 

The construction method known as the ``simplex method'' introduced in \cite{AshPos2013} uses a polynomial vector field, given by equations (2) in \cite{AshPos2013}, such that all equilibria in the network lie on the coordinate axes. The plane containing two consecutive equilibria in a heteroclinic cycle contains a trajectory in the heteroclinic connection between the equilibria. These trajectories together with the equilibria they connect form a quasi-simple cycle. Therefore, the simplex method is especially suitable when addressing our question (Q3).

\section{Minimal complete realisations}\label{sec:completions}
In this section we address (Q1) and (Q2) from the introduction.

In what follows all constructions can be achieved by the simplex method. In particular, all transverse eigenvalues can (and must!) be chosen as negative. This produces networks for which asymptotic stability can be obtained with the results in \cite{PodCasLab2020}.

In this section we consider only digraphs $G=(V,A)$ consisting of two cycles. 
Let $k$ and $\ell$ denote the number of vertices of each cycle.
We can compare the length of the cycles by counting the number of vertices (or edges). If $k \leq \ell$, we say that the $k$-cycle is the shorter cycle, and the $\ell$-cycle is the longer cycle.

The simplex method can be used only if the number of vertices is greater than or equal to three. Hence, throughout this section we assume $k,\ell \geq 3$. A digraph consisting of two cycles always has a unique distribution vertex, which may be distinct from or coincide with the unique collection vertex. If they are distinct, then there is a common edge (or a sequence of common edges). For example, this is the case for the graph of the Kirk-Silber network in \cite{KirSil1994} where $k = \ell =3$, and for that of the House network in \cite{CasLoh2016a} where $k = 3$ and $\ell =4$, see Figure~\ref{fig-KS-Bowtie-House} (left/right, respectively). If collection and distribution vertices coincide, the network has only one common vertex. An example for this case is the Bowtie network in \cite{AshPos2013}, see Figure~\ref{fig-KS-Bowtie-House} (middle).

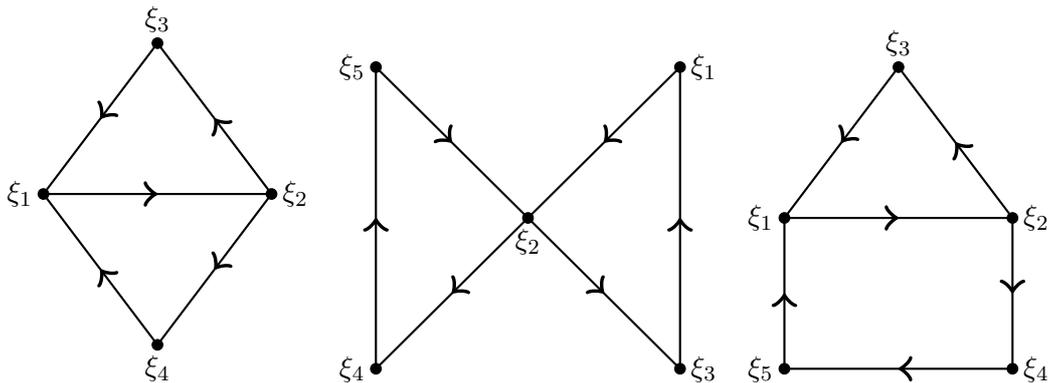
\begin{figure}[!htb]
 \centerline{
\begin{tikzpicture}[xscale=1, yscale=1]
\filldraw[black] (0,2) circle (2pt) node[above] {\small$\xi_3$};
\filldraw[black] (-1.5,0) circle (2pt) node[left] {\small$\xi_1$};
\filldraw[black] (1.5,0) circle (2pt) node[right] {\small$\xi_2$};
\filldraw[black] (0,-2) circle (2pt) node[below] {\small$\xi_4$};
\path (1.5,0) edge[black,thick, middlearrow={>}{}{}] (0,2);
\path (0,2) edge[black,thick, middlearrow={>}{}{}] (-1.5,0);
\path (-1.5,0) edge[black,thick, middlearrow={>}{}{}] (1.5,0);
\path (1.5,0) edge[black,thick, middlearrow={>}{}{}] (0,-2);
\path (0,-2) edge[black,thick, middlearrow={>}{}{}] (-1.5,0);
\end{tikzpicture}
\begin{tikzpicture}[xscale=1, yscale=1]
\filldraw[black] (2,-2) circle (2pt) node[right] {\small$\xi_3$};
\filldraw[black] (2,2) circle (2pt) node[right] {\small$\xi_1$};
\filldraw[black] (0,0) circle (2pt) node[below] {\small$\xi_2$};
\filldraw[black] (-2,-2) circle (2pt) node[left] {\small$\xi_4$};
\filldraw[black] (-2,2) circle (2pt) node[left] {\small$\xi_5$};
\path (0,0) edge[black,thick, middlearrow={>}{}{}] (2,-2);
\path (2,-2) edge[black,thick, middlearrow={>}{}{}] (2,2);
\path (2,2) edge[black,thick, middlearrow={>}{}{}] (0,0);
\path (0,0) edge[black,thick, middlearrow={>}{}{}] (-2,-2);
\path (-2,-2) edge[black,thick, middlearrow={>}{}{}] (-2,2);
\path (-2,2) edge[black,thick, middlearrow={>}{}{}] (0,0);
\end{tikzpicture}
\begin{tikzpicture}[xscale=1, yscale=1]
\filldraw[black] (0,2) circle (2pt) node[above] {\small$\xi_3$};
\filldraw[black] (-1.5,0) circle (2pt) node[left] {\small$\xi_1$};
\filldraw[black] (1.5,0) circle (2pt) node[right] {\small$\xi_2$};
\filldraw[black] (1.5,-2) circle (2pt) node[right] {\small$\xi_4$};
\filldraw[black] (-1.5,-2) circle (2pt) node[left] {\small$\xi_5$};
\path (1.5,0) edge[black,thick, middlearrow={>}{}{}] (0,2);
\path (0,2) edge[black,thick, middlearrow={>}{}{}] (-1.5,0);
\path (-1.5,0) edge[black,thick, middlearrow={>}{}{}] (1.5,0);
\path (1.5,0) edge[black,thick, middlearrow={>}{}{}] (1.5,-2);
\path (1.5,-2) edge[black,thick, middlearrow={>}{}{}] (-1.5,-2);
\path (-1.5,-2) edge[black,thick, middlearrow={>}{}{}] (-1.5,0);
\end{tikzpicture}
}
 \caption{Graphs for the Kirk-Silber network (left), the Bowtie (middle), and the House (right).\label{fig-KS-Bowtie-House}}
 \end{figure}

\begin{proposition}\label{prop:minimal}
Let $G=(V,A)$ be a digraph consisting of two cycles.
Let $k$ and $\ell$ denote the number of vertices of each cycle with $3\leq k \leq \ell$ and $\ell>3$.
Suppose there is no more than one common edge.
Then the minimal number of edges that have to be added to obtain a complete realisation of $G$ is $k-1$.
\end{proposition}

\begin{proof}
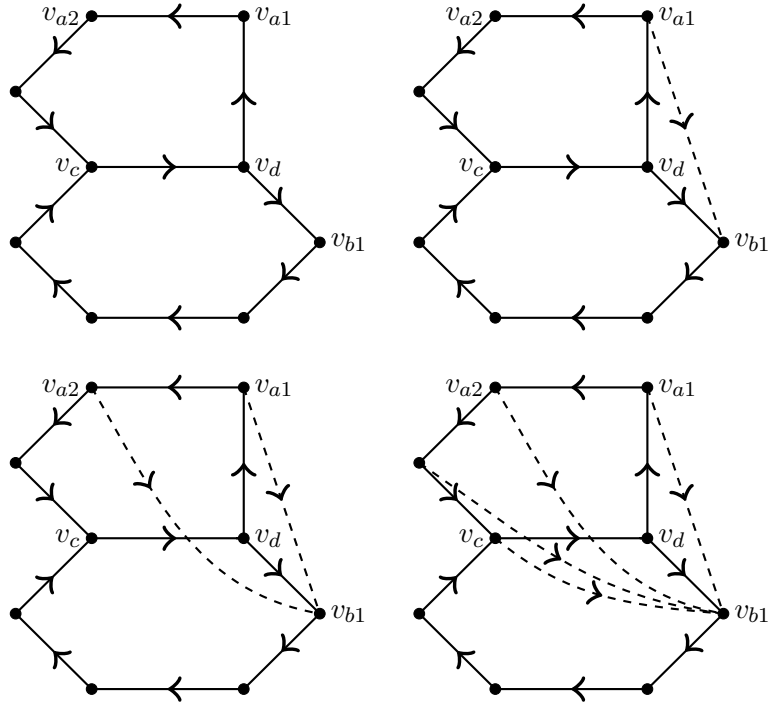
\begin{figure}[!htb]
\centerline{
\begin{tikzpicture}[xscale=1, yscale=1]
\filldraw[black] (0,0) circle (2pt) node[left] {\small$v_c$};
\filldraw[black] (2,0) circle (2pt) node[right] {\small$v_d$};
\filldraw[black] (2,2) circle (2pt) node[right] {\small$v_{a1}$};
\filldraw[black] (0,2) circle (2pt) node[left] {\small$v_{a2}$};
\filldraw[black] (-1,1) circle (2pt);
\filldraw[black] (-1,-1) circle (2pt);
\filldraw[black] (0,-2) circle (2pt);
\filldraw[black] (2,-2) circle (2pt);
\filldraw[black] (3,-1) circle (2pt) node[right] {\small$v_{b1}$};
\path (0,0) edge[out=0,in=0, black,thick, middlearrow={>}{}{}] (1.9,0);
\path (2,0) edge[out=90,in=-90, black,thick, middlearrow={>}{}{}] (2,2);
\path (2,2) edge[out=180,in=0, black,thick, middlearrow={>}{}{}] (0,2);
\path (0,2) edge[out=225,in=45, black,thick, middlearrow={>}{}{}] (-1,1);
\path (-1,1) edge[out=-45,in=135, black,thick, middlearrow={>}{}{}] (0,0);
\path (2,0) edge[out=-45,in=135, black,thick, middlearrow={>}{}{}] (3,-1);
\path (3,-1) edge[out=225,in=45, black,thick, middlearrow={>}{}{}] (2,-2);
\path (2,-2) edge[out=180,in=0, black,thick, middlearrow={>}{}{}] (0,-2);
\path (0,-2) edge[out=135,in=-45, black,thick, middlearrow={>}{}{}] (-1,-1);
\path (-1,-1) edge[out=225,in=45, black,thick, middlearrow={>}{}{}] (0,0);
\end{tikzpicture}
\begin{tikzpicture}[xscale=1, yscale=1]
\filldraw[black] (0,0) circle (2pt) node[left] {\small$v_c$};
\filldraw[black] (2,0) circle (2pt) node[right] {\small$v_d$};
\filldraw[black] (2,2) circle (2pt) node[right] {\small$v_{a1}$};
\filldraw[black] (0,2) circle (2pt) node[left] {\small$v_{a2}$};
\filldraw[black] (-1,1) circle (2pt);
\filldraw[black] (-1,-1) circle (2pt);
\filldraw[black] (0,-2) circle (2pt);
\filldraw[black] (2,-2) circle (2pt);
\filldraw[black] (3,-1) circle (2pt) node[right] {\small$v_{b1}$};
\path (0,0) edge[out=0,in=0, black,thick, middlearrow={>}{}{}] (1.9,0);
\path (2,0) edge[out=90,in=-90, black,thick, middlearrow={>}{}{}] (2,2);
\path (2,2) edge[out=180,in=0, black,thick, middlearrow={>}{}{}] (0,2);
\path (0,2) edge[out=225,in=45, black,thick, middlearrow={>}{}{}] (-1,1);
\path (-1,1) edge[out=-45,in=135, black,thick, middlearrow={>}{}{}] (0,0);
\path (2,0) edge[out=-45,in=135, black,thick, middlearrow={>}{}{}] (3,-1);
\path (3,-1) edge[out=225,in=45, black,thick, middlearrow={>}{}{}] (2,-2);
\path (2,-2) edge[out=180,in=0, black,thick, middlearrow={>}{}{}] (0,-2);
\path (0,-2) edge[out=135,in=-45, black,thick, middlearrow={>}{}{}] (-1,-1);
\path (-1,-1) edge[out=225,in=45, black,thick, middlearrow={>}{}{}] (0,0);
\path (2,2) edge[black,thick, dashed, middlearrow={>}{}{}] (3,-1);
\end{tikzpicture}
}
\vspace{.5cm}
\centerline{
\begin{tikzpicture}[xscale=1, yscale=1]
\filldraw[black] (0,0) circle (2pt) node[left] {\small$v_c$};
\filldraw[black] (2,0) circle (2pt) node[right] {\small$v_d$};
\filldraw[black] (2,2) circle (2pt) node[right] {\small$v_{a1}$};
\filldraw[black] (0,2) circle (2pt) node[left] {\small$v_{a2}$};
\filldraw[black] (-1,1) circle (2pt);
\filldraw[black] (-1,-1) circle (2pt);
\filldraw[black] (0,-2) circle (2pt);
\filldraw[black] (2,-2) circle (2pt);
\filldraw[black] (3,-1) circle (2pt) node[right] {\small$v_{b1}$};
\path (0,0) edge[out=0,in=0, black,thick, middlearrow={>}{}{}] (1.9,0);
\path (2,0) edge[out=90,in=-90, black,thick, middlearrow={>}{}{}] (2,2);
\path (2,2) edge[out=180,in=0, black,thick, middlearrow={>}{}{}] (0,2);
\path (0,2) edge[out=225,in=45, black,thick, middlearrow={>}{}{}] (-1,1);
\path (-1,1) edge[out=-45,in=135, black,thick, middlearrow={>}{}{}] (0,0);
\path (2,0) edge[out=-45,in=135, black,thick, middlearrow={>}{}{}] (3,-1);
\path (3,-1) edge[out=225,in=45, black,thick, middlearrow={>}{}{}] (2,-2);
\path (2,-2) edge[out=180,in=0, black,thick, middlearrow={>}{}{}] (0,-2);
\path (0,-2) edge[out=135,in=-45, black,thick, middlearrow={>}{}{}] (-1,-1);
\path (-1,-1) edge[out=225,in=45, black,thick, middlearrow={>}{}{}] (0,0);
\path (2,2) edge[black,thick, dashed, middlearrow={>}{}{}] (3,-1);
\begin{scope}[very thick,decoration={markings, mark=at position 0.35 with {\arrow[scale=2]{>}}}] \path (0,2) edge[out=-60,in=170, black,thick, dashed, postaction={decorate}] (3,-1);
\end{scope}
\end{tikzpicture}
\begin{tikzpicture}[xscale=1, yscale=1]
\filldraw[black] (0,0) circle (2pt) node[left] {\small$v_c$};
\filldraw[black] (2,0) circle (2pt) node[right] {\small$v_d$};
\filldraw[black] (2,2) circle (2pt) node[right] {\small$v_{a1}$};
\filldraw[black] (0,2) circle (2pt) node[left] {\small$v_{a2}$};
\filldraw[black] (-1,1) circle (2pt);
\filldraw[black] (-1,-1) circle (2pt);
\filldraw[black] (0,-2) circle (2pt);
\filldraw[black] (2,-2) circle (2pt);
\filldraw[black] (3,-1) circle (2pt) node[right] {\small$v_{b1}$};
\path (0,0) edge[out=0,in=0, black,thick, middlearrow={>}{}{}] (1.9,0);
\path (2,0) edge[out=90,in=-90, black,thick, middlearrow={>}{}{}] (2,2);
\path (2,2) edge[out=180,in=0, black,thick, middlearrow={>}{}{}] (0,2);
\path (0,2) edge[out=225,in=45, black,thick, middlearrow={>}{}{}] (-1,1);
\path (-1,1) edge[out=-45,in=135, black,thick, middlearrow={>}{}{}] (0,0);
\path (2,0) edge[out=-45,in=135, black,thick, middlearrow={>}{}{}] (3,-1);
\path (3,-1) edge[out=225,in=45, black,thick, middlearrow={>}{}{}] (2,-2);
\path (2,-2) edge[out=180,in=0, black,thick, middlearrow={>}{}{}] (0,-2);
\path (0,-2) edge[out=135,in=-45, black,thick, middlearrow={>}{}{}] (-1,-1);
\path (-1,-1) edge[out=225,in=45, black,thick, middlearrow={>}{}{}] (0,0);
\path (2,2) edge[black,thick, dashed, middlearrow={>}{}{}] (3,-1);
\begin{scope}[very thick,decoration={markings, mark=at position 0.35 with {\arrow[scale=2]{>}}}] \path (0,2) edge[out=-60,in=170, black,thick, dashed, postaction={decorate}] (3,-1);
\end{scope}
\path (-1,1) edge[out=-35,in=173, black,thick, dashed, middlearrow={>}{}{}] (3,-1);
\path (0,0) edge[out=-40,in=175, black,thick, dashed, middlearrow={>}{}{}] (3,-1);
\end{tikzpicture}
}
 \caption{For the graph at the top left the simplex method produces and incomplete network since $v_d$ has out-degree 2. An edge can be added connecting $v_{a1}$ to $v_{b1}$ in order to create a $\Delta$-clique to contain the unstable manifold of the distribution node corresponding to $v_d$ (top right). The problem of having out-degree two without a $\Delta$-clique repeats at $v_{a1}$ and can be solved by the addition of an edge $[v_{a2} \to v_{b1}]$ (bottom left). The graph corresponding to the complete network appears at the bottom right.\label{fig:prop-minimal}}
\end{figure}
 
The proof is constructive and we start by containing the 2-dimensional unstable manifold of the distribution node. In order to produce a $\Delta$-clique that will contain this unstable manifold, we add an edge connecting the two vertices, one on each cycle, immediately after the distribution vertex. Orient this edge from the shorter to the longer cycle. In Figure~\ref{fig:prop-minimal}, this is the connection $[v_{a1} \to v_{b1}]$, forcing a new positive eigenvalue at the equilibrium corresponding to $v_{a1}$, when the simplex method is applied. 

This makes $v_{a1}$ into a distribution vertex which corresponds to an equilibrium with a 2-dimensional unstable manifold, not contained in the network. We repeat the previous step at $v_{a1}$ and create a $\Delta$-clique by adding an edge connecting the vertex after $v_{a1}$ and the same vertex in the longer cycle, $v_{b1}$. Again direct this edge from the shorter to the longer cycle. This is $[v_{a2} \to v_{b1}]$ in Figure~\ref{fig:prop-minimal}.

This process finishes when all vertices in the shorter cycle are connected to $v_{b1}$ which amounts to the addition of $k-1$ edges, all directed towards the same vertex, $v_{b1}$. In Figure~\ref{fig:prop-minimal} this is when the vertices $v_c$, $v_d$ and $v_{b1}$ form a $\Delta$-clique.

Now every distribution vertex in the extended graph has out-degree 2 and is the b-point of a $\Delta$-clique, which captures the entire unstable manifold of the corresponding equilibrium. Thus, the network created through the simplex method is complete.

Minimality is guaranteed by adding edges from the shorter to the longer cycle. As can be seen in Figure~\ref{fig:prop-minimal} this choice of orientation pushes the location of the distribution vertex that is not yet a b-point of a $\Delta$-clique along the shorter cycle. The process finishes once a distribution vertex is created that connects to the original one. In Figure~\ref{fig:prop-minimal} this happens when $v_c$ becomes a distribution vertex. Choosing different orientation of the added edges in the process increases the number of additional edges.
\end{proof}

In order to apply the simplex method we need cycles with at least three equilibria. In the result above we have restricted to $\ell>3$, so we now briefly comment on what happens for $k=\ell=3$, denoting by $m$ the number of common edges between the two cycles:
\begin{itemize}
	\item If $m=0$, then we have the graph of the Bowtie network~\cite{CasLoh2016a} and the minimal number of new edges required for a complete realisation is $2=k-1=\ell-1$ (just as in Proposition~\ref{prop:minimal}).
	\item If $m=1$, then we have the graph of the Kirk and Silber network~\cite{KirSil1994} and the minimal number of new edges required for a complete realisation is $1=k-2=\ell-2$.
	\item The case $m=2$ is not possible, because it would require the distribution vertex to be connected to the collection vertex by two separate edges, which contradicts our general assumptions on the graphs and networks we consider. 
\end{itemize}

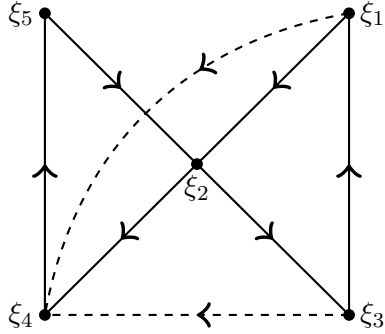
\begin{figure}[!htb]
 \centerline{
\begin{tikzpicture}[xscale=1, yscale=1]
\filldraw[black] (2,-2) circle (2pt) node[right] {\small$\xi_3$};
\filldraw[black] (2,2) circle (2pt) node[right] {\small$\xi_1$};
\filldraw[black] (0,0) circle (2pt) node[below] {\small$\xi_2$};
\filldraw[black] (-2,-2) circle (2pt) node[left] {\small$\xi_4$};
\filldraw[black] (-2,2) circle (2pt) node[left] {\small$\xi_5$};
\path (0,0) edge[black,thick, middlearrow={>}{}{}] (2,-2);
\path (2,-2) edge[black,thick, middlearrow={>}{}{}] (2,2);
\path (2,2) edge[black,thick, middlearrow={>}{}{}] (0,0);
\path (0,0) edge[black,thick, middlearrow={>}{}{}] (-2,-2);
\path (-2,-2) edge[black,thick, middlearrow={>}{}{}] (-2,2);
\path (-2,2) edge[black,thick, middlearrow={>}{}{}] (0,0);
\path (2,-2) edge[black,thick, dashed, middlearrow={>}{}{}] (-2,-2);
\begin{scope}[very thick,decoration={markings, mark=at position 0.35 with {\arrow[scale=2]{>}}}] 
\path (2,2) edge[out=190,in=80, black,thick, dashed, postaction={decorate}] (-2,-2);
\end{scope}
\end{tikzpicture}
}
 \caption{Graph for the complete realisation of the Bowtie network. The added edges are indicated by dashed lines. Adding $[\xi_3 \rightarrow \xi_4]$ to contain $W^u(\xi_2)$ creates a 2-dimensional unstable manifold at $\xi_3$. This can be contained by adding the edge $[\xi_1 \rightarrow \xi_4]$. Note that the opposite orientation for this edge leads to a 2-dimensional unstable manifold at $\xi_4$ leading to a non-minimal complete realisation.\label{fig-Bowtie-completion}}
 \end{figure}
 
 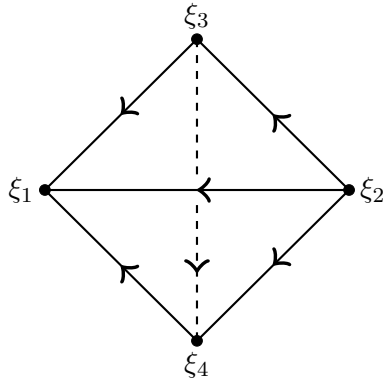
\begin{figure}[!htb]
 \centerline{
\begin{tikzpicture}[xscale=1, yscale=1]
\filldraw[black] (0,2) circle (2pt) node[above] {\small$\xi_3$};
\filldraw[black] (-2,0) circle (2pt) node[left] {\small$\xi_1$};
\filldraw[black] (2,0) circle (2pt) node[right] {\small$\xi_2$};
\filldraw[black] (0,-2) circle (2pt) node[below] {\small$\xi_4$};
\path (2,0) edge[black,thick, middlearrow={>}{}{}] (0,2);
\path (0,2) edge[black,thick, middlearrow={>}{}{}] (-2,0);
\path (-2,0) edge[black,thick, middlearrow={>}{}{}] (2,0);
\path (2,0) edge[black,thick, middlearrow={>}{}{}] (0,-2);
\path (0,-2) edge[black,thick, middlearrow={>}{}{}] (-2,0);
\path (0,2) edge[black,thick, dashed] (0,0.2);
\path (0,-0.2) edge[black,thick, dashed, middlearrow={>}{}{}] (0,-2);
\end{tikzpicture}
}
 \caption{Graph for the complete realisation of the Kirk-Silber network. The added edge is indicated by dashed lines. Adding $[\xi_3 \rightarrow \xi_4]$ to contain $W^u(\xi_2)$ completes the network. The orientation of this edge is irrelevant.\label{fig-KS-completion}}
 \end{figure}

It is worthwhile noting that the same graph can be realised as a complete network by adding edges from the longer to the shorter cycle in exactly the same way as in the proof above. If $k=3$ and there is one common edge, this process stops when reaching the vertex before the collection vertex after the addition of $\ell -2$ edges. Otherwise, it stops at the vertex before the distribution vertex (which coincides with the collection vertex if there is one common edge), resulting in $\ell-1$ new edges. Although this produces a complete network, the number of edges is usually not minimal.

\begin{proposition}\label{prop:more-common}
Let $G=(V,A)$ be a digraph consisting of two cycles.
Let $k$ and $\ell$ denote the number of vertices of each cycle with $3 \leq k \leq \ell$.
Suppose there are $m>1$ common edges.
Then the minimal number of edges that have to be added to obtain a complete realisation of $G$ is given by
\begin{enumerate}[(i)]
	\item $k-1$ if $k>m+2$,
	\item $\min(\ell-(m+1),k-1)$ if $k=m+2$,
	\item $\min(\ell-(m+2),k-1)$ if $k=m+1$.
\end{enumerate}
\end{proposition}

\begin{proof}
Note that we always have $k \geq m+1$, so (i)-(iii) cover all cases.

The proof proceeds as the previous one, by adding edges either from the shorter to the longer cycle or vice-versa.  The process ends differently in these two cases.

If we add edges from the shorter to the longer cycle, then reaching the collection vertex does not create the last $\Delta$-clique. The process must be continued along the common edges until the last vertex before the distribution vertex, see Figure~\ref{fig:prop-more-common}, where as before all vertices in the shorter cycle have to be connected to $v_{b1}$. This creates $k-1$ new edges.

If we add edges from the longer to the shorter cycle, there are several possibilities: For (i) note that the process only stops at the vertex before the distribution vertex, just like when we are connecting from short to long, resulting in $\ell-1$ new edges. Since $\ell \geq k$, the minimal number of edges is $k-1$.

For (ii), the process stops one vertex before the collection vertex, because all new edges lead to the vertex before the collection vertex (there is only one vertex in the shorter cycle that does not belong to the longer cycle). This results in $\ell-(m+1)$ new edges.

Finally, for (iii) the process stops two vertices before the collection vertex, because all new edges lead to the collection vertex, resulting in $\ell-(m+2)$ new edges. The case $m=2$, $k=3$, $\ell =4$ is complete without the addition of edges.
\end{proof}

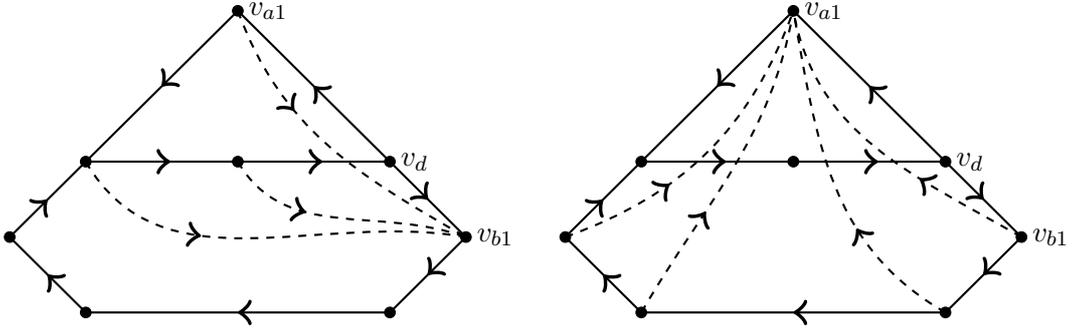
\begin{figure}[!htb]
\centerline{
\begin{tikzpicture}[xscale=1, yscale=1]
\filldraw[black] (0,0) circle (2pt);
\filldraw[black] (2,0) circle (2pt);
\filldraw[black] (4,0) circle (2pt) node[right] {\small$v_d$};
\filldraw[black] (2,2) circle (2pt) node[right] {\small$v_{a1}$};
\filldraw[black] (-1,-1) circle (2pt);
\filldraw[black] (0,-2) circle (2pt);
\filldraw[black] (4,-2) circle (2pt);
\filldraw[black] (5,-1) circle (2pt) node[right] {\small$v_{b1}$};
\path (0,0) edge[out=0,in=0, black,thick, middlearrow={>}{}{}] (1.9,0);
\path (2,0) edge[out=0,in=0, black,thick, middlearrow={>}{}{}] (3.9,0);
\path (4,0) edge[out=135,in=-45, black,thick, middlearrow={>}{}{}] (2,2);
\path (2,2) edge[out=225,in=45, black,thick, middlearrow={>}{}{}] (0,0);
\path (4,0) edge[out=-45,in=135, black,thick, middlearrow={>}{}{}] (5,-1);
\path (5,-1) edge[out=225,in=45, black,thick, middlearrow={>}{}{}] (4,-2);
\path (4,-2) edge[out=180,in=0, black,thick, middlearrow={>}{}{}] (0,-2);
\path (0,-2) edge[out=135,in=-45, black,thick, middlearrow={>}{}{}] (-1,-1);
\path (-1,-1) edge[out=225,in=45, black,thick, middlearrow={>}{}{}] (0,0);
\begin{scope}[very thick,decoration={markings, mark=at position 0.35 with {\arrow[scale=2]{>}}}] 
\path (2,2) edge[out=-70,in=150, black,thick, dashed, postaction={decorate}] (5,-1);
\path (0,0) edge[out=-60,in=170, black,thick, dashed, postaction={decorate}] (5,-1);
\path (2,0) edge[out=-60,in=160, black,thick, dashed, postaction={decorate}] (5,-1);
\end{scope}
\end{tikzpicture}
\begin{tikzpicture}[xscale=1, yscale=1]
\filldraw[black] (0,0) circle (2pt);
\filldraw[black] (2,0) circle (2pt);
\filldraw[black] (4,0) circle (2pt) node[right] {\small$v_d$};
\filldraw[black] (2,2) circle (2pt) node[right] {\small$v_{a1}$};
\filldraw[black] (-1,-1) circle (2pt);
\filldraw[black] (0,-2) circle (2pt);
\filldraw[black] (4,-2) circle (2pt);
\filldraw[black] (5,-1) circle (2pt) node[right] {\small$v_{b1}$};
\path (0,0) edge[out=0,in=0, black,thick, middlearrow={>}{}{}] (1.9,0);
\path (2,0) edge[out=0,in=0, black,thick, middlearrow={>}{}{}] (3.9,0);
\path (4,0) edge[out=135,in=-45, black,thick, middlearrow={>}{}{}] (2,2);
\path (2,2) edge[out=225,in=45, black,thick, middlearrow={>}{}{}] (0,0);
\path (4,0) edge[out=-45,in=135, black,thick, middlearrow={>}{}{}] (5,-1);
\path (5,-1) edge[out=225,in=45, black,thick, middlearrow={>}{}{}] (4,-2);
\path (4,-2) edge[out=180,in=0, black,thick, middlearrow={>}{}{}] (0,-2);
\path (0,-2) edge[out=135,in=-45, black,thick, middlearrow={>}{}{}] (-1,-1);
\path (-1,-1) edge[out=225,in=45, black,thick, middlearrow={>}{}{}] (0,0);
\begin{scope}[very thick,decoration={markings, mark=at position 0.35 with {\arrow[scale=2]{>}}}] 
\path (5,-1) edge[out=150,in=-80, black,thick, dashed, postaction={decorate}] (2,2);
\path (4,-2) edge[out=150,in=-80, black,thick, dashed, postaction={decorate}] (2,2);
\path (0,-2) edge[out=60,in=-100, black,thick, dashed, postaction={decorate}] (2,2);
\path (-1,-1) edge[out=20,in=-110, black,thick, dashed, postaction={decorate}] (2,2);
\end{scope}
\end{tikzpicture}
}
\caption{Both initial graphs (drawn with solid lines) are equal and such that $k=4$, $\ell = 7$, and $m=2$. On the left we construct a graph that can be realised as a complete network by adding edges from the shorter to the longer cycle. We add $k-1=3$ edges. On the right we add $\ell-(m+1)=4$ edges from the longer to the shorter cycle.
 \label{fig:prop-more-common}}
\end{figure}

See Figure~\ref{fig:prop-more-common}. 
Which choice of direction for the new edges yields the smaller number of necessary new edges depends on how long the longer cycle is compared to the number of common edges in both cycles.

\begin{corollary}
The first expression in the minimum in cases (ii) and (iii) above is the smaller one if and only if $\ell \leq 2(m+1)$.
\end{corollary}

There are other ways of achieving a complete realisation, namely by choosing edges from one cycle to the other in alternating ways. This is never minimal and falls outside the scope of the present article.

\section{Positive transverse eigenvalues for cycles in complete realisations}\label{sec:stability}
In this section we address (Q3) from the introduction. In contrast to what we have done so far, we now focus on the one-dimensional connections in the coordinate planes, which do not always correspond to the full intersection of the respective invariant manifolds. When restricting to these connections, the simplex method produces quasi-simple heteroclinic cycles whose stability can be studied using results in \cite{GarCas2019} or \cite{Pod2012}. For these results the eigenvalues of the Jacobian matrix at the equilibria in the network are important.

In this section we show how different complete realisations influence the signs of these eigenvalues. This provides enough information to apply the following first steps towards a study of stability:

\begin{itemize}
	\item[(i)] Given a digraph consisting of two cycles as before, choose the cycle for which to minimize the number of positive transverse eigenvalues.
	\item[(ii)] Based on the choice above order the coordinates (which may be thought of as a labeling of the vertices in the graph) in a way that does not affect the stability in the construction, but simplifies the necessary calculations.
	\item[(iii)] Apply the simplex method to create a heteroclinic network using the choices made above.
\end{itemize}

We give a brief overview of how stability can be studied by constructing a Poincar\'e return map to cross sections near each equilibrium. Let $H_j^{\inn}$ and $H_j^{\out}$ be cross sections near $\xi_j$, respectively, along the incoming connection to $\xi_j$ and the outgoing connection to $\xi_j$. Note that it suffices to consider one incoming and one outgoing cross-section per equilibrium, since we are looking at a single cycle, not the whole network. Define the {\em local map} near $\xi_j$ by $\phi_j: H_j^{\inn} \to H_j^{\out}$ obtained by integrating the linearised flow at $\xi_j$. Define also the {\em global map} between $\xi_j$ and $\xi_{j+1}$ by $\psi_j: H_j^{\out} \to H_{j+1}^{\inn}$. Denoting by $n_t$ the number of transverse eigenvalues at each equilibrium, we can restrict the cross sections to an $(n_t+1)$-dimensional space where we can use coordinates $(w,\boldsymbol{z})$ with $w \in \R$, related to the expanding direction, and $\boldsymbol{z} \in \R^{n_t}$, related to the transverse directions. Following also \cite{Pod2012} we use logarithmic coordinates of the form $\boldsymbol{\eta} = (\ln w, \ln z_1, \hdots, \ln z_{n_t})$. In these coordinates the map $g_j \equiv \psi_j \circ \phi_j: H_j^{\inn} \to H_{j+1}^{\inn}$ becomes linear and can be represented by a matrix $\mathcal{M}_{j}\boldsymbol{\eta}=M_{j}\boldsymbol{\eta}+F_{j}$. The Poincar\'e return maps to each cross section are obtained by calculating the products of such matrices. It can be shown that only the products of the matrices $M_j$ matter, see \cite{Pod2012}. Therefore, for our purposes, we are interested in the basic transition matrices $M_j$ given by

\begin{equation}\label{eq:Mj}
M_{j}=A_{j}\left[\begin{array}{ccccc}
b_{j,1} & 0 & 0 & \ldots & 0\\
b_{j,2} & 1 & 0 & \ldots & 0\\
b_{j,3} & 0 & 1 & \ldots & 0\\
. & . & . & \ldots & .\\
b_{j,N} & 0 & 0 & \ldots & 1
\end{array}\right],
\end{equation}
where the entries depend on the eigenvalues at $\xi_j$ as follows
$$
b_{j,1}=\frac{c_{j}}{e_{j}}, \;\; b_{j,s+1}=-\frac{t_{j,s}}{e_{j}}, \;\; s=1,\ldots,n_{t}, \;\; j=1,\ldots,m.
$$
The matrix $A_j$ is a permutation matrix which becomes the identity matrix when the global map is the identity. Otherwise, Proposition 4.1.10 in \cite{Gar2018} implies that $A_j$ has the following block-diagonal form:
\begin{align*}
A_j =  \left( \begin{array}{c|c}
A & 0 \\\hline
0 & I \\
\end{array}\right)
\mbox{\hspace{1cm}}
A =  \left( \begin{array}{ccccc}
0 & 0 & \ldots & 0 & 1 \\
1 & 0 & \ldots & 0 & 0 \\
0 & 1 & \ldots & 0 & 0 \\
\ldots & \ldots & \ldots & \ldots & \ldots \\
0 & 0 & \ldots & 1 & 0 \\
\end{array}\right)
\end{align*}
The dimension of the matrix $A$ is given by the length of the cycle minus 2, and the identity block corresponds to the transverse directions (to the cycle).

Note that we always have $b_{j,1}>0$, but other entries of $M_j$ may be negative when transverse eigenvalues to the cycle are positive. These positive transverse eigenvalues contribute to the instability of the cycle. The permutation matrix $A_j$ can be thought of as spreading the effect of such a positive entry throughout the transition matrix products for the return map. Therefore, it seems impossible to formulate a general statement on stability of cycles in complete realisations. However, some results can be obtained from the graph structure.

In our construction, the number of positive transverse eigenvalues in a cycle is related to the number of added edges. The actual number of positive transverse eigenvalues also depends on the number of common connections as explained in the following subsections. Each cycle has at least one positive transverse eigenvalue, namely at the distribution node.

\subsection{No common connection/edge}
In networks without a common connection, the two cycles only share a single equilibrium. This is both the distribution and collection node. The simplest example of such a network realised by the simplex method is the Bowtie network.

\begin{figure}[!htb]
 \centerline{
\begin{tikzpicture}[xscale=1, yscale=1]
\filldraw[black] (2,-2) circle (2pt) node[right] {\small$\xi_3$};
\filldraw[black] (2,2) circle (2pt) node[right] {\small$\xi_1$};
\filldraw[black] (0,0) circle (2pt) node[below] {\small$\xi_2$};
\filldraw[black] (-2,-2) circle (2pt) node[left] {\small$\xi_4$};
\filldraw[black] (-4,0) circle (2pt) node[left] {\small$\xi_5$};
\filldraw[black] (-2,2) circle (2pt) node[left] {\small$\xi_6$};
\path (0,0) edge[black,thick, middlearrow={>}{}{}] (2,-2);
\path (2,-2) edge[black,thick, middlearrow={>}{}{}] (2,2);
\path (2,2) edge[black,thick, middlearrow={>}{}{}] (0,0);
\path (0,0) edge[black,thick, middlearrow={>}{}{}] (-2,-2);
\path (-2,-2) edge[black,thick, middlearrow={>}{}{}] (-4,0);
\path (-4,0) edge[black,thick, middlearrow={>}{}{}] (-2,2);
\path (-2,2) edge[black,thick, middlearrow={>}{}{}] (0,0);
\path (-2,-2) edge[black,thick, dashed, middlearrow={>}{}{}] (2,-2);
\begin{scope}[very thick,decoration={markings, mark=at position 0.35 with {\arrow[scale=2]{>}}}] 
\path (-4,0) edge[out=-30,in=170, black,thick, dashed, postaction={decorate}] (2,-2);
\path (-2,2) edge[out=-80,in=160, black,thick, dashed, postaction={decorate}] (2,-2);
\end{scope}
\end{tikzpicture}
}
 \caption{Graph for the complete realisation of a network with no common connection. The collection node coincides with the distribution node and is labelled  $\xi_2$. In this network, $k=3$ and $\ell=4$. The added edges are dashed lines and go from the longer to the shorter cycle, $[\xi_1 \rightarrow \xi_2 \rightarrow \xi_3]$, so that this complete realisation is not minimal.
 \label{fig:no-common-connection}}
 \end{figure}
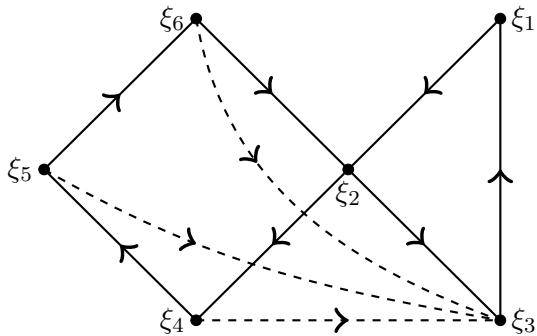

\begin{proposition}\label{prop:stab-no-common-edge}
Consider a digraph consisting of two cycles that have only one common vertex. Then for each cycle there is a complete realisation such that the only positive transverse eigenvalue of this cycle is at the distribution node.
\end{proposition}

\begin{proof}
The proof works similarly to that of Proposition~\ref{prop:minimal} and we illustrate it through the example in Figure~\ref{fig:no-common-connection}: There the two-dimensional unstable manifold of the distribution node $\xi_2$ is contained through the construction of a $\Delta$-clique by adding an edge between the next two vertices, $\xi_3$ and $\xi_4$ in the example. Depending on the choice of direction for this edge we then proceed to add more edges connecting to the same vertex. The process finishes (i.e.\ the network resulting from applying the simplex method to the extended graph is complete) once the vertex just before the distribution vertex is reached, in the case of Figure~\ref{fig:no-common-connection} this is $\xi_6$. Since the new edges all point to the same vertex, the cycle containing this vertex (on the right in Figure~\ref{fig:no-common-connection}) has no edges added that point away from it. This means that the only positive transverse eigenvalue is the one at the distribution node. Choosing the opposite direction for the first new edge ($[\xi_3 \to \xi_4]$ instead of $[\xi_4 \to \xi_3]$) yields a different complete realisation where the other cycle has only one positive transverse eigenvalue (also at the distribution node, $\xi_2$).
\end{proof}

Notice that in the proof above more vertices become distribution vertices as edges are added. However, for a given cycle ($[\xi_1 \rightarrow \xi_2 \rightarrow \xi_3]$ in Figure~\ref{fig:no-common-connection}), only the original distribution vertex corresponds to an equilibrium where the Jacobian matrix has a positive transverse eigenvalue. This is the minimal number of positive eigenvalues when a cycle is part of a heteroclinic network.

\begin{corollary}
This realisation is minimal if and only if all added edges point towards the longer cycle.
\end{corollary}

\begin{proof}
This follows directly from Proposition~\ref{prop:minimal}.
\end{proof}

\subsection{One or more common connection/edge}
Let $m\geq 1$ be the number of common connections. Since we use the simplex method both cycles have to be of length at least three. Thus, if $m=1$, we have $k \geq 3=m+2$ for the shorter cycle. If $m>1$, then $k=m+1$ is possible. We treat the cases $k\leq m+2$ and $k>m+2$ separately in the next two propositions.

\begin{figure}[!htb]
 \centerline{
\begin{tikzpicture}[xscale=1, yscale=1]
\filldraw[black] (0,0) circle (2pt);
\filldraw[black] (2,0) circle (2pt);
\filldraw[black] (4,0) circle (2pt) node[right] {\small$\xi_d$};
\filldraw[black] (2,2) circle (2pt);
\filldraw[black] (-1,-1) circle (2pt);
\filldraw[black] (0,-2) circle (2pt);
\filldraw[black] (4,-2) circle (2pt);
\filldraw[black] (5,-1) circle (2pt);
\path (0,0) edge[out=0,in=0, black,thick, middlearrow={>}{}{}] (1.9,0);
\path (2,0) edge[out=0,in=0, black,thick, middlearrow={>}{}{}] (3.9,0);
\path (4,0) edge[out=135,in=-45, black,thick, middlearrow={>}{}{}] (2,2);
\path (2,2) edge[out=225,in=45, black,thick, middlearrow={>}{}{}] (0,0);
\path (4,0) edge[out=-45,in=135, black,thick, middlearrow={>}{}{}] (5,-1);
\path (5,-1) edge[out=225,in=45, black,thick, middlearrow={>}{}{}] (4,-2);
\path (4,-2) edge[out=180,in=0, black,thick, middlearrow={>}{}{}] (2.5,-2);
\path (1.5,-2) edge[out=180,in=0, black,thick, middlearrow={>}{}{}] (0,-2);
\path (0,-2) edge[out=135,in=-45, black,thick, middlearrow={>}{}{}] (-1,-1);
\path (-1,-1) edge[out=225,in=45, black,thick, middlearrow={>}{}{}] (0,0);
\filldraw[black] (2,-1.85) circle (0pt) node[below]  {$\hdots$};
\begin{scope}[very thick,decoration={markings, mark=at position 0.35 with {\arrow[scale=2]{>}}}] 
\path (5,-1) edge[out=150,in=-80, black,thick, dashed, postaction={decorate}] (2,2);
\path (4,-2) edge[out=150,in=-80, black,thick, dashed, postaction={decorate}] (2,2);
\path (0,-2) edge[out=60,in=-100, black,thick, dashed, postaction={decorate}] (2,2);
\path (-1,-1) edge[out=20,in=-110, black,thick, dashed, postaction={decorate}] (2,2);
\end{scope}
\end{tikzpicture}
}
\caption{A complete realisation so that the shorter cycle has only one positive transverse eigenvalue at the distribution node, $\xi_d$. Here $m=2$, $k=m+2=4$. The realisation is not minimal.
 \label{fig:no-extra-positive}}
\end{figure}
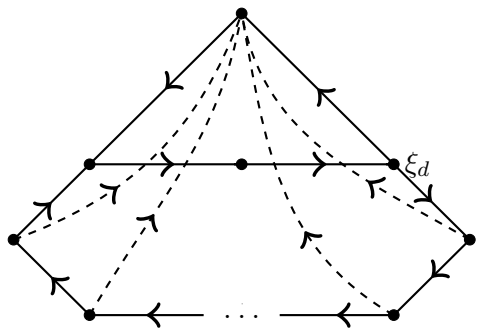

\begin{proposition}\label{prop:no-extra-positive}
Consider a digraph consisting of two cycles that have $m \geq 1$ common edges. If the shorter cycle is of length at most $m+2$, there is a complete realisation such that the only positive transverse eigenvalue of this cycle is at the distribution node.
\end{proposition}

\begin{proof}
The proof works analogously to that of Proposition~\ref{prop:stab-no-common-edge}. We choose the direction of the added edges such that they point towards the shorter cycle, resulting in a situation as shown in Figure~\ref{fig:no-extra-positive}. No edges pointing away from the shorter cycle are needed because there is only one vertex in the shorter cycle that does not belong to both cycles, and thus, there is a $\Delta$-clique involving the collection vertex and the two vertices just before it (one in either cycle).
\end{proof}

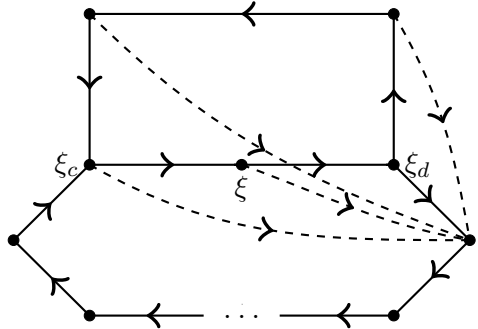
\begin{figure}[!htb]
 \centerline{
\begin{tikzpicture}[xscale=1, yscale=1]
\filldraw[black] (0,0) circle (2pt) node[left] {\small$\xi_c$};
\filldraw[black] (2,0) circle (2pt) node[below] {\small$\xi$};
\filldraw[black] (4,0) circle (2pt) node[right] {\small$\xi_d$};
\filldraw[black] (4,2) circle (2pt);
\filldraw[black] (0,2) circle (2pt);
\filldraw[black] (-1,-1) circle (2pt);
\filldraw[black] (0,-2) circle (2pt);
\filldraw[black] (4,-2) circle (2pt);
\filldraw[black] (5,-1) circle (2pt);
\path (0,0) edge[out=0,in=0, black,thick, middlearrow={>}{}{}] (1.9,0);
\path (2,0) edge[out=0,in=0, black,thick, middlearrow={>}{}{}] (3.9,0);
\path (4,0) edge[out=90,in=-90, black,thick, middlearrow={>}{}{}] (4,2);
\path (4,2) edge[out=180,in=0, black,thick, middlearrow={>}{}{}] (0,2);
\path (0,2) edge[out=-90,in=90, black,thick, middlearrow={>}{}{}] (0,0);
\path (4,0) edge[out=-45,in=135, black,thick, middlearrow={>}{}{}] (5,-1);
\path (5,-1) edge[out=225,in=45, black,thick, middlearrow={>}{}{}] (4,-2);
\path (4,-2) edge[out=180,in=0, black,thick, middlearrow={>}{}{}] (2.5,-2);
\path (1.5,-2) edge[out=180,in=0, black,thick, middlearrow={>}{}{}] (0,-2);
\path (0,-2) edge[out=135,in=-45, black,thick, middlearrow={>}{}{}] (-1,-1);
\path (-1,-1) edge[out=225,in=45, black,thick, middlearrow={>}{}{}] (0,0);
\filldraw[black] (2,-1.85) circle (0pt) node[below]  {$\hdots$};
\path (4,2) edge[out=-60,in=100, black,thick, dashed, middlearrow={>}{}{}] (5,-1);
\path (0,2) edge[out=-45,in=160, black,thick, dashed, middlearrow={>}{}{}] (5,-1);
\path (0,0) edge[out=-30,in=180, black,thick, dashed, middlearrow={>}{}{}] (5,-1);
\path (2,0) edge[out=-20,in=170, black,thick, dashed, middlearrow={>}{}{}] (5,-1);
\end{tikzpicture}
}
\caption{A complete realisation so that the longer cycle has three positive transverse eigenvalues, one at the distribution node, $\xi_d$, another at the collection node, $\xi_c$, and the third at the equilibrium $\xi$, belonging to the common connections. Here $m=2$, $k=m+3=5$. The realisation is minimal.
 \label{fig:more-positive}}
\end{figure}

\begin{proposition}
Consider a digraph consisting of two cycles that have $m \geq 1$ common edges. If the shorter cycle is of length greater than $m+2$, then for each cycle there is a complete realisation such that this cycle has $m+1$ positive transverse eigenvalues, namely at all the common equilibria.
\end{proposition}

\begin{proof}
If the shorter cycle has length greater than $m+2$, each cycle has at least two vertices that do not belong to the other cycle. Repeating the same process as in the previous proofs, we have to keep adding edges until all vertices in one cycle are directly connected to the same vertex in the other cycle. The edges at the common vertices produce positive transverse eigenvalues. Their number is $m+1$. This is illustrated in Figure~\ref{fig:more-positive}. Edges in the opposite direction give a different complete realisation with $m+1$ positive transverse eigenvalues for the other cycle.
\end{proof}

\begin{corollary}
This realisation is minimal if and only if all added edges point towards the longer cycle.
\end{corollary}

\begin{proof}
This follows from Propositions~\ref{prop:minimal} and \ref{prop:more-common}, see Figure~\ref{fig:more-positive}.
\end{proof}

These results provide information on the number of positive transverse eigenvalues that occur in complete realisations of the type considered here. This is necessary for a detailed investigation of the stability of individual cycles through the established methods, e.g.\ calculating stability indices, see e.g.~\cite{PodAsh2011, GarCas2019, Pod2012}.

As an example consider the House network, detailed calculations are given in Appendix~\ref{app:house}. From Proposition~\ref{prop:no-extra-positive} and the block diagonal form of the transition matrices it follows that the stability of the cycle of length three can be obtained by using results by Podvigina and Ashwin ~\cite{PodAsh2011} in $\R^4$. We show in Appendix~\ref{app:house} that the results for $\R^4$ are not sufficient for determining the stability of the cycle of length four.

\section{Networks with more than two cycles}\label{sec:further}

In this section we address two ways in which a graph can contain more than two cycles and comment on the extent to which our method of obtaining a complete realisation can still be applied. These are

\begin{enumerate}
	\item[(i)] there is more than one distribution vertex,
	\item[(ii)] there is a distribution vertex with more than two outgoing edges.
\end{enumerate}

According to Lemma~\ref{lem:all-b-points} we focus on cases with at least one distribution vertex that is not the b-point of a $\Delta$-clique.

The ``donut'' graph has four cycles. There are two collection vertices and two distribution vertices, labelled $\xi_{c1}$, $\xi_{c2}$ and $\xi_{d1}$, $\xi_{d2}$ in Figure~\ref{fig:donut-simple}, respectively. A complete realisation is obtained by adding two edges, one to create a $\Delta$-clique at $\xi_{d1}$ and another at $\xi_{d2}$. This can be done in any order and the orientation of the added edges does not matter.

\begin{figure}[!htb]
\centerline{
\begin{tikzpicture}[xscale=1, yscale=1]
\filldraw[black] (0,0) circle (2pt) node[left] {\small$\xi_{d1}$};
\filldraw[black] (2,0) circle (2pt) node[above] {\small$\xi_{c2}$};
\filldraw[black] (3,0.5) circle (2pt);
\filldraw[black] (3,-0.5) circle (2pt);
\filldraw[black] (4,0) circle (2pt) node[above] {\small$\xi_{d2}$};
\filldraw[black] (6,0) circle (2pt) node[right] {\small$\xi_{c1}$};
\filldraw[black] (3,2) circle (2pt);
\filldraw[black] (3,-2) circle (2pt);
\path (2,0) edge[black,thick, middlearrow={>}{}{}] (0,0);
\path (3,0.5) edge[black,thick, middlearrow={>}{}{}] (2,0);
\path (3,-0.5) edge[black,thick, middlearrow={>}{}{}] (2,0);
\path (4,0) edge[black,thick, middlearrow={>}{}{}] (3,0.5);
\path (4,0) edge[black,thick, middlearrow={>}{}{}] (3,-0.5);
\path (6,0) edge[black,thick, middlearrow={>}{}{}] (4,0);
\path (0,0) edge[black,thick, middlearrow={>}{}{}] (3,2);
\path (3,2) edge[black,thick, middlearrow={>}{}{}] (6,0);
\path (0,0) edge[black,thick, middlearrow={>}{}{}] (3,-2);
\path (3,-2) edge[black,thick, middlearrow={>}{}{}] (6,0);
\path (3,0.5) edge[black,thick, dashed] (3,-0.5);
\path (3,2) edge[out=235,in=115,black,thick, dashed] (3,-2);
\end{tikzpicture}
}
\caption{The ``donut'' graph has two distribution vertices and consists of four cycles, corresponding to the choices up-down from each distribution vertex. The two dashed edges allow for a complete realisation by introducing the two $\Delta$-cliques necessary to capture the unstable manifold of each distribution node. The orientation of these edges is irrelevant. \label{fig:donut-simple}}
\end{figure}
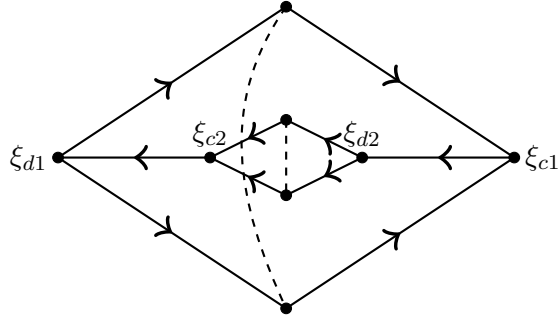

Now consider a donut graph with more than one vertex between a distribution and a collection vertex, and thus more edges. Our procedure of Section~\ref{sec:completions} can still be applied until a vertex of out-degree three is created in the process, see Figure~\ref{fig:donut}. This puts us in situation (ii) from above and it is unclear how to proceed, as explained below.

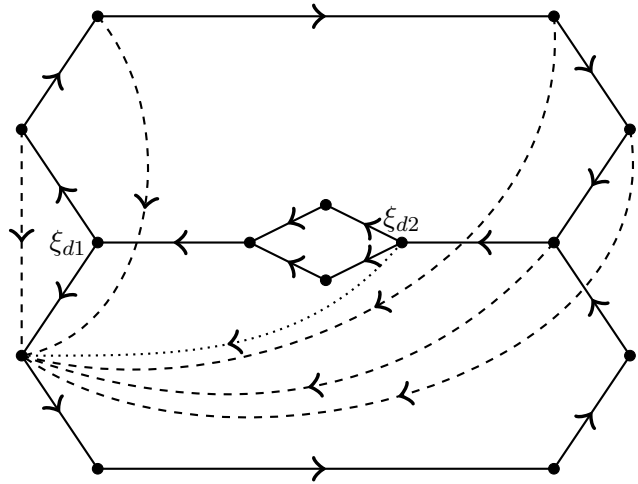
\begin{figure}[!htb]
\centerline{
\begin{tikzpicture}[xscale=1, yscale=1]
\coordinate (1) at (0,0);
\coordinate (2) at (2,0);
\coordinate (3) at (3,0.5);
\coordinate (4) at (3,-0.5);
\coordinate (5) at (4,0);
\coordinate (6) at (6,0);
\coordinate (7) at (-1,1.5);
\coordinate (8) at (0,3);
\coordinate (9) at (6,3);
\coordinate (10) at (7,1.5);
\coordinate (11) at (-1,-1.5);
\coordinate (12) at (0,-3);
\coordinate (13) at (6,-3);
\coordinate (14) at (7,-1.5);
\filldraw[black] (1) circle (2pt) node[left] {\small$\xi_{d1}$};
\filldraw[black] (2) circle (2pt);
\filldraw[black] (3) circle (2pt);
\filldraw[black] (4) circle (2pt);
\filldraw[black] (5) circle (2pt) node[above] {\small$\xi_{d2}$};
\filldraw[black] (6) circle (2pt);
\filldraw[black] (7) circle (2pt);
\filldraw[black] (8) circle (2pt);
\filldraw[black] (9) circle (2pt);
\filldraw[black] (10) circle (2pt);
\filldraw[black] (11) circle (2pt);
\filldraw[black] (12) circle (2pt);
\filldraw[black] (13) circle (2pt);
\filldraw[black] (14) circle (2pt);
\path (2,0) edge[black,thick, middlearrow={>}{}{}] (0,0);
\path (3,0.5) edge[black,thick, middlearrow={>}{}{}] (2,0);
\path (3,-0.5) edge[black,thick, middlearrow={>}{}{}] (2,0);
\path (4,0) edge[black,thick, middlearrow={>}{}{}] (3,0.5);
\path (4,0) edge[black,thick, middlearrow={>}{}{}] (3,-0.5);
\path (6,0) edge[black,thick, middlearrow={>}{}{}] (4,0);
\path (1) edge[black,thick, middlearrow={>}{}{}] (7);
\path (7) edge[black,thick, middlearrow={>}{}{}] (8);
\path (8) edge[black,thick, middlearrow={>}{}{}] (9);
\path (9) edge[black,thick, middlearrow={>}{}{}] (10);
\path (10) edge[black,thick, middlearrow={>}{}{}] (6);
\path (1) edge[black,thick, middlearrow={>}{}{}] (11);
\path (11) edge[black,thick, middlearrow={>}{}{}] (12);
\path (12) edge[black,thick, middlearrow={>}{}{}] (13);
\path (13) edge[black,thick, middlearrow={>}{}{}] (14);
\path (14) edge[black,thick, middlearrow={>}{}{}] (6);
%
\path (7) edge[black,thick, dashed, middlearrow={>}{}{}] (11);
\path (8) edge[out=-50,in=10,black,thick, dashed, middlearrow={>}{}{}] (11);
\path (9) edge[out=-85,in=-15,black,thick, dashed, middlearrow={>}{}{}] (11);
\path (10) edge[out=-80,in=-35,black,thick, dashed, middlearrow={>}{}{}] (11);
\path (6) edge[out=-130,in=-20,black,thick, dashed, middlearrow={>}{}{}] (11);
\path (5) edge[out=-130,in=0,black,thick, dotted, middlearrow={>}{}{}] (11);
\end{tikzpicture}
}
\caption{This ``donut'' graph has more vertices than the previous one. It still has two distribution vertices and consists of four cycles, corresponding to the choices up-down from each distribution vertex. The addition of the dotted edge creates a vertex of out-degree three at $\xi_{d2}$. \label{fig:donut}}
\end{figure}

For a graph with a distribution vertex of out-degree three, the simplex method generates an equilibrium with a three-dimensional unstable manifold. Unlike in the case of a $\Delta$-clique, now the addition of an edge in the graph does not guarantee that this entire unstable manifold limits to one of the other equilibria in the network. An attempt to still capture the unstable manifold is to add edges connecting the vertices with incoming edges from the distribution vertex. This can lead to either a $\Delta$-clique or a cycle of length three between these vertices, see Figure~\ref{fig-3d-unstable-manifold}.

In the case of a $\Delta$-clique, $\xi_3$ becomes a sink in the four-dimensional subspace spanned by the directions containing all the equilibria, and this may lead to a three-dimensional connection from $\xi_d$ to $\xi_3$, a two-dimensional connection from $\xi_d$ to $\xi_2$ and a one-dimensional connection from $\xi_d$ to $\xi_1$.

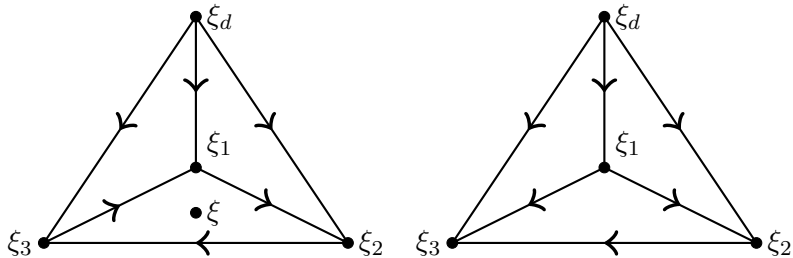
\begin{figure}[!htb]
 \centerline{
\begin{tikzpicture}[xscale=1, yscale=1]
\filldraw[black] (0,0) circle (2pt) node[above right] {\small$\xi_1$};
\filldraw[black] (0,2) circle (2pt) node[right] {\small$\xi_d$};
\filldraw[black] (-2,-1) circle (2pt) node[left] {\small$\xi_3$};
\filldraw[black] (2,-1) circle (2pt) node[right] {\small$\xi_2$};
\filldraw[black] (0,-0.6) circle (2pt) node[right] {\small$\xi$};
\path (0,2) edge[out=-90,in=90, black,thick, middlearrow={>}{}{}] (0,0);
\path (0,2) edge[black,thick, middlearrow={>}{}{}] (-2,-1);
\path (0,2) edge[black,thick, middlearrow={>}{}{}] (2,-1);
\path (0,0) edge[black,thick, middlearrow={>}{}{}] (2,-1);
\path (2,-1) edge[black,thick, middlearrow={>}{}{}] (-2,-1);
\path (-2,-1) edge[black,thick, middlearrow={>}{}{}] (0,0);
\end{tikzpicture}
%
\begin{tikzpicture}[xscale=1, yscale=1]
\filldraw[black] (0,0) circle (2pt) node[above right] {\small$\xi_1$};
\filldraw[black] (0,2) circle (2pt) node[right] {\small$\xi_d$};
\filldraw[black] (-2,-1) circle (2pt) node[left] {\small$\xi_3$};
\filldraw[black] (2,-1) circle (2pt) node[right] {\small$\xi_2$};
\path (0,2) edge[out=-90,in=90, black,thick, middlearrow={>}{}{}] (0,0);
\path (0,2) edge[black,thick, middlearrow={>}{}{}] (-2,-1);
\path (0,2) edge[black,thick, middlearrow={>}{}{}] (2,-1);
\path (0,0) edge[black,thick, middlearrow={>}{}{}] (2,-1);
\path (2,-1) edge[black,thick, middlearrow={>}{}{}] (-2,-1);
\path (0,0) edge[black,thick, middlearrow={>}{}{}] (-2,-1);
\end{tikzpicture}
}
 \caption{A subgraph that causes the simplex method to generate a three-dimensional unstable manifold. The edges between the vertices $\xi_1, \xi_2, \xi_3$, which have incoming edges from the distribution vertex $\xi_d$, form a cycle of length three (left) or a $\Delta$-clique (right). In the left case an additional equilibrium $\xi$ appears in the realisation: its existence is guaranteed by Poincar\'e-Bendixson and it does not belong to the network.\label{fig-3d-unstable-manifold}}
 \end{figure}

In the case of a cycle, each equilibrium has a non-trivial unstable manifold within the four-dimensional subspace and thus none of the equilibria is a sink. This typically leads to the existence of an extra equilibrium as sketched in Figure~\ref{fig-3d-unstable-manifold}. Since the simplex method produces attracting radial directions, this implies that none of the stable manifolds of $\xi_j$, $j \in \{1,2,3\}$ can contain a three-dimensional subset of the unstable manifold of $\xi_d$. It is possible that a large part of the unstable manifold of $\xi_d$ limits to the cycle between $\xi_1, \xi_2, \xi_3$, producing a depth two heteroclinic connection, see Definition 2.22 in~\cite{AshFie1999}. An example where such a situation is unavoidable is shown in Figure~\ref{fig:no-realisation}, where there is just one distribution node (of out-degree two) which is not the b-point of a $\Delta$-clique. The graph in the figure can be made strongly connected by adding suitable edges.

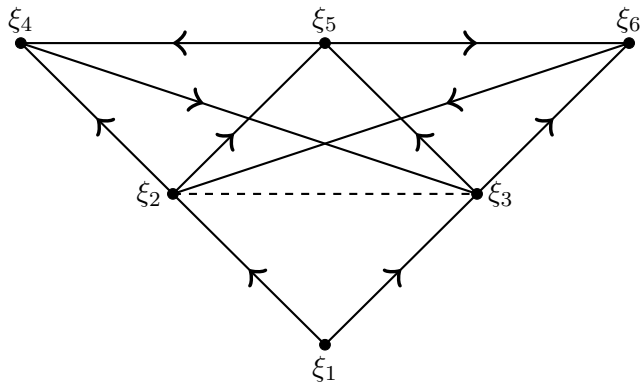
\begin{figure}[!htb]
\centerline{
\begin{tikzpicture}[xscale=1, yscale=1]
\coordinate (1) at (0,0);
\coordinate (2) at (-2,2);
\coordinate (3) at (2,2);
\coordinate (4) at (-4,4);
\coordinate (5) at (0,4);
\coordinate (6) at (4,4);
\filldraw[black] (1) circle (2pt) node[below] {\small$\xi_{1}$};
\filldraw[black] (2) circle (2pt) node[left] {\small$\xi_{2}$};
\filldraw[black] (3) circle (2pt) node[right] {\small$\xi_{3}$};
\filldraw[black] (4) circle (2pt) node[above] {\small$\xi_{4}$};
\filldraw[black] (5) circle (2pt) node[above] {\small$\xi_{5}$};
\filldraw[black] (6) circle (2pt) node[above] {\small$\xi_{6}$};
\begin{scope}[very thick,decoration={markings, mark=at position 0.4 with {\arrow[scale=2]{>}}}] 
\path (2) edge[black,thick,postaction={decorate}] (5);
\path (3) edge[black,thick,postaction={decorate}] (5);
\path (4) edge[black,thick,postaction={decorate}] (3);
\path (6) edge[black,thick,postaction={decorate}] (2);
\end{scope}
\path (1) edge[black,thick, middlearrow={>}{}{}] (2);
\path (1) edge[black,thick, middlearrow={>}{}{}] (3);
\path (2) edge[black,thick, middlearrow={>}{}{}] (4);
\path (3) edge[black,thick, middlearrow={>}{}{}] (6);
\path (5) edge[black,thick, middlearrow={>}{}{}] (4);
\path (5) edge[black,thick, middlearrow={>}{}{}] (6);
\path (2) edge[black,thick, dashed] (3);
\end{tikzpicture}
}
\caption{A graph with four distribution vertices labelled $\xi_1$, $\xi_2$, $\xi_3$, and $\xi_5$. The dashed edge between $\xi_2$ and $\xi_3$ is required to contain the 2-dimensional unstable manifold of $\xi_1$ in the realisation. Regardless of the orientation of this edge, a cycle of length three appears leading to a connection of depth two in the network. This cycle consists of either $[\xi_2 \rightarrow \xi_5 \rightarrow \xi_6 \rightarrow \xi_2]$ or $[\xi_3 \rightarrow \xi_5 \rightarrow \xi_4 \rightarrow \xi_3]$.\label{fig:no-realisation}}
\end{figure}

These considerations allow us to conclude that not every graph has a complete realisation.

On the other hand, there are also graphs that do not require the addition of further edges in order for the heteroclinic network resulting from the simplex method to be complete.
\begin{lemma}\label{lem:all-b-points}
Let $G$ be a strongly connected digraph such that every distribution vertex has out-degree two and is the b-point of a $\Delta$-clique in $G$. Then the heteroclinic network obtained by realizing $G$ with the simplex method is complete.
\end{lemma}

\begin{proof}
The two-dimensional manifolds of the distribution nodes are contained in the $\Delta$-cliques and are therefore part of the heteroclinic network.
\end{proof}

\section{Concluding remarks}\label{sec:conclusion}

We have established a method for obtaining complete heteroclinic networks from graphs with two cycles. This method does not add equilibria to the network and allows for some control over the stability of each heteroclinic cycle. 

Our remarks in Section~\ref{sec:further} show that applying this method to more general graphs inevitably encounters very particular features. One immediate extension of our work would be to consider joining two complete heteroclinic networks rather than two heteroclinic cycles. This is beyond the scope of this article and perhaps better addressed case by case, as interest arises. Some of our tools may, however, still be of use.

Our considerations show that we often encounter a cycle for which one of the following is true: (i) there is a single equilibrium where the Jacobian matrix has a unique positive transverse eigenvalue, all other transverse eigenvalues (at all other equilibria) are negative; or (ii), there is one positive transverse eigenvalue of the Jacobian matrix at every equilibrium in the cycle.

From the stability perspective, this prompts the following questions: assuming that all quantities not specified above can be chosen as desired, how stable can a quasi-simple heteroclinic cycle be in these two cases? We expect that in case (i) such a cycle can be e.a.s. In case (ii) we anticipate the orientation of the transverse directions corresponding to the positive eigenvalues to play a role in the stability of the cycle.	

These questions are left to be addressed elsewhere.

\paragraph{Acknowledgements:}
The first author was partially supported by CMUP, member of LASI, which is financed by national funds through FCT -- Funda\c{c}\~ao para a Ci\^encia e a Tecnologia, I.P., with reference UIDB/00144/2025.

Part of the work was done when the authors visited each other's institutions and the hospitality of the University of Porto and the University of Hamburg is gratefully acknowledged.

\appendix
\section{List of networks with two cycles}\label{app:list}
We provide here a list of some heteroclinic networks consisting of precisely two cycles that have been studied in the literature. Not all of these can be thought of as created by the simplex method (indicated in the penultimate column), nevertheless this list might help to put our work into proper context.

Recall that $k$ and $\ell$ are the number of equilibria in the shorter and longer cycle, respectively, while $m$ is the number of common connections between the two cycles.

For a systematic overview of all graphs associated with heteroclinic networks in $\R^4$, not limited to only two cycles, see~\cite{PodLoh2019}.

\begin{table}[H]
\begin{center}
\begin{tabular}{ c | c | c | c | c | c | c }
$m$ & $k$ & $\ell$ & Network & Space & Simplex? & Reference  \\ \hline \hline 
0 & 2 & 2 & $(C_2^-, C_2^-)$ & $\R^6$  & no & \cite{Gar2018} \\ \hline
0 & 3 & 3 & Bowtie & $\R^5$  & yes & \cite{CasLoh2016a} \\ \hline
1 & 2 & 2 & $(B_2^+, B_2^+)$ & $\R^4$ & no & \cite{CasLoh2016b} \\ \hline
1 & 3 & 3 & Kirk-Silber & $\R^4$ & yes & \cite{KirSil1994} \\ \hline
1 & 3 & 4 & House & $\R^5$ & yes & \cite{CasLoh2016a} \\ \hline
1 & 4 & 4 &  & $\R^6$ & yes & \cite{Pod2023} \\ \hline
2 & 3 & 4 & $(B_3^-, C_4^-)$ & $\R^4$ & yes & \cite{Bra1994} \\ \hline
2 & 6 & 6 &  & $(\R/2\pi\Z)^4$ & no & \cite{BickLoh2019} \\
\end{tabular}
\end{center}
\caption{Heteroclinic networks consisting of precisely two cycles.} \label{table:networks}
\end{table}

\section{Calculations for the House network}\label{app:house}

The graph for the House network is shown in Figure~\ref{fig-KS-Bowtie-House} (right). Since it has five vertices, it can be realised by the simplex method in $\R^5$. We focus on two complete realisations which both require the addition of two new edges in the graph and are therefore minimal in the sense of this paper. For both of the following two cases we outline how the study of stability of cycles in the network may or may not be reduced to the study of a lower-dimensional cycle for which results are available in the literature:

\begin{itemize}
	\item[Case (a)] Adding two edges, from both $\xi_4$ and $\xi_5$ to $\xi_3$. The new edges are oriented towards the cycle of length three and, thus, we are interested in the stability of this cycle.
	\item[Case (b)] Adding two edges, from both $\xi_3$ and $\xi_1$ to $\xi_4$. The new edges are oriented towards the cycle of length four and, thus, we are interested in the stability of this cycle.
\end{itemize}

We denote by $-c_{ij}<0$ the contracting eigenvalue, by $e_{ij}>0$ the expanding eigenvalue and by $\pm t_{ij} \gtrless 0$ the transverse eigenvalue at $\xi_i$ in the direction of $\xi_j$. The two exceptions occur at $\xi_1$ where $c_{13}$ and $c_{15}$ represent either contracting or transverse eigenvalues, and at $\xi_2$ where $e_{23}$ and $e_{24}$ represent expanding or transverse eigenvalues, depending on which cycle is of interest. All $c_{ij}$, $e_{ij}$, and $t_{ij}$ are positive.
We assume throughout this appendix that all the transverse eigenvalues are negative, except when explicitly necessary (and explicitly stated).

As usual the stability is determined by the properties of the return map to cross sections near the equilibria of the cycle. These can be obtained from basic transition matrices and their products which express the dynamics with respect to logarithmic coordinates, see \cite{KruMel1995}. A basic transition matrix describes the flow between two consecutive cross sections, $H^{\inn,j-1}_j \rightarrow H^{\inn,j}_{j+1}$, respectively, across a connecting trajectory from $\xi_{j-1}$ to $\xi_j$ and across a connecting trajectory from $\xi_{j}$ to $\xi_{j+1}$.

\medbreak

Case (a). The transverse eigenvalues for the cycle of length three are all negative, except for $e_{24}$ at $\xi_2$. For this cycle the basic transition matrices are of the form \eqref{eq:Mj} where $A_j$ is the identity, $N=3$, and all entries are non-negative except $b_{2,2} = -e_{24}/e_{23} < 0$ in $M_2$. For stability calculations we need the products of basic transition matrices that correspond to a full return around the cycle. Due to the lower triangular structure of the matrices $M_j$, which is preserved under multiplication, these are of the form
$$
M=\left[\begin{array}{ccc}
\alpha_1 & 0 &  0\\
\alpha_2 & 1 & 0\\
\alpha_3 & 0 & 1
\end{array}\right],
$$
where the $\alpha_j$ are sums of products of some quotients of eigenvalues such that $\alpha_1, \alpha_3 >0$ but potentially $\alpha_2 < 0$.
The top left $2 \times 2$ block is as in Subsection 4.2.1 in \cite{PodAsh2011} concerning the $B_3^-$ cycles in $\R^4$. Because of the triangular structure, the extra dimension does not interfere with the results -- except where the eigenvector for the biggest eigenvalue is concerned, since this is now a vector in $\R^3$. Let $(u_1,u_2,u_3)$ be this eigenvector and $\lambda = \frac{c_{13}c_{21}c_{32}}{e_{12}e_{23}e_{31}}$. It must be $\lambda > 1$ for stability by Lemma 5 in \cite{Pod2012}. The third equation from the definition of an eigenvector for a matrix $M$ as above is
$$
\alpha_3 u_1 +u_3 = \lambda u_3 \quad \Leftrightarrow \quad u_3 = \dfrac{\alpha_3}{\lambda -1} u_1.
$$
So, the sign of the coordinate $u_3$ is the same as that of $u_1$ when $\lambda>1$. Therefore, the stability of the cycle of length three in the House network can be obtained from the results in \cite{PodAsh2011}.

\bigbreak

Case (b). In this instance for the cycle of length four both $\xi_1$ and $\xi_2$ have one positive transverse eigenvalue, respectively, $t_{14}$ and $e_{23}$.
The basic transition matrices have the form (in this case $A_j$ is a non-trivial permutation exchanging the first two rows of a matrix):
$$
M_1=\left[\begin{array}{ccc}
-\dfrac{t_{14}}{e_{12}} & 1 &  0\\
 & & \\
\dfrac{c_{15}}{e_{12}} & 0 & 0\\
 & & \\
\dfrac{c_{13}}{e_{12}} & 0 & 1
\end{array}\right], \mbox{  }
M_2=\left[\begin{array}{ccc}
\dfrac{t_{25}}{e_{24}} & 1 &  0\\
 & & \\
\dfrac{c_{21}}{e_{24}} & 0 & 0\\
 & & \\
-\dfrac{e_{23}}{e_{24}} & 0 & 1
\end{array}\right],  \mbox{  }
$$
$$
M_4=\left[\begin{array}{ccc}
\dfrac{t_{41}}{e_{45}} & 1 &  0\\
 & & \\
\dfrac{c_{42}}{e_{45}} & 0 & 0\\
 & & \\
\dfrac{t_{43}}{e_{45}} & 0 & 1
\end{array}\right], 
\mbox{ and }
M_5=\left[\begin{array}{ccc}
\dfrac{t_{52}}{e_{51}} & 1 &  0\\
 & & \\
\dfrac{c_{54}}{e_{51}} & 0 & 0\\
 & & \\
\dfrac{t_{53}}{e_{51}} & 0 & 1
\end{array}\right].
$$
The products of these matrices take the form 
$$
M=\left[\begin{array}{ccc}
\alpha_1 & \beta_1 &  0\\
\alpha_2 & \beta_2 & 0\\
\alpha_3 & \beta_3 & 1
\end{array}\right],
$$
where $\beta_1, \beta_2 > 0$ but the other parameters may have either sign. Tedious calculations show that $\det{M}$ is the ratio between contracting and expanding eigenvalues leading to the usual necessary condition for stability given by $\det{M}>1$. Because of the form of the matrices the condition that the eigenvalue with maximum absolute value is real and greater than one can be checked from the results in \cite{PodAsh2011} concerning the $C_4^-$ cycle, using the top left $2 \times 2$ block. Denote this eigenvalue by $\lambda$. The eigenvector condition in Lemma 5 in \cite{Pod2012} that all the coordinates of the eigenvector $(u_1,u_2,u_3)$ for the eigenvalue $\lambda$ must have the same sign for stability involves not only the $2\times2$ top left block but also the equation
$$
\alpha_3 u_1 + \beta_3 u_2 +u_3 = \lambda u_3 \Leftrightarrow (\lambda -1)u_3 = \alpha_3 u_1 + \beta_3 u_2.
$$
Whether this provides the same or a different sign for $u_3$ depends on the signs and magnitudes of $\alpha_3$ and $\beta_3$ and no longer follows from existing calculations.

\end{document}